\documentclass[]{imsart}

\RequirePackage{amsthm,amsmath,amsfonts,amssymb,multirow,bbm}
\RequirePackage[numbers]{natbib}
\RequirePackage[colorlinks,citecolor=blue,urlcolor=blue]{hyperref}
\RequirePackage{graphicx}

\startlocaldefs
\theoremstyle{plain}

\newtheorem{theorem}{Theorem}[section]
\newtheorem{lemma}[theorem]{Lemma}
\newtheorem{prop}[theorem]{Proposition}
\newtheorem{coro}[theorem]{Corollary}

\theoremstyle{remark}

\newtheorem{example}{Example}
\newtheorem{pseudocode}{Pseudocode}

\newtheorem{remark}[theorem]{Remark}
\newtheorem{obs}[theorem]{Observation}
\newtheorem*{assumption}{Assumption}
\newtheorem*{condition}{Condition}

\endlocaldefs

\begin{document}
\begin{frontmatter}
\title{k-Sample inference \\ via Multimarginal Optimal Transport}
\runtitle{k-Sample inference via Multimarginal Optimal Transport}

\begin{aug}
\author[A]{\fnms{Natalia}~\snm{Kravtsova}\ead[label=e1]{kravtsova.2@osu.edu}}
\address[A]{Department of Mathematics,
The Ohio State University\printead[presep={,\ }]{e1}}

\end{aug}

\begin{abstract}
This paper proposes a Multimarginal Optimal Transport ($MOT$) approach for simultaneously comparing $k\geq 2$ measures supported on  finite subsets of $\mathbb{R}^d$, $d \geq 1$. We derive asymptotic distributions of the optimal value of the empirical $MOT$ program under the null hypothesis that all $k$ measures are same, and the alternative hypothesis that at least two measures are different. We use these results to construct the test of the null hypothesis and provide consistency and power guarantees of this $k$-sample test. We consistently estimate asymptotic distributions using bootstrap, and propose a low complexity linear program to approximate the test cut-off. We demonstrate the advantages of our approach on synthetic and real datasets, including the real data on cancers in the United States in 2004 - 2020.
\end{abstract}



\end{frontmatter}
\tableofcontents

\section{Introduction}
The $k$-sample inference concerns with simultaneously comparing several probability measures. The classical question of this inference is to determine whether $k\geq 2$ groups of observed data points have the same underlying probability distribution, i.e.  to test 

\begin{equation}\label{eq:hypothesis}
	\begin{aligned}
		& H_0: \mu_1 = \cdots = \mu_k \\
		& H_a: \mu_i \neq \mu_j \text{ for some } 1 \leq i < j \leq k
	\end{aligned}
\end{equation}
This testing problem has a long history in statistics, with classical rank-based tests for univariate data \cite{Conover,Kiefer,Scholz,Zhang} to recent extension \cite{Deb} using multivariate ranks \cite{Chernozhukov,Hallin_ranks,HallinHlubinka}, to graph \cite{Mukhopadhyay}, distance \cite{Rizzo} and kernel \cite{Sejdinovic,Kim} based methods. Direct applications of testing hypotheses in \eqref{eq:hypothesis} include simultaneously comparing gene expression profiles to assess presence of disease \cite{Zhan}, assessing differences in chronic disease levels based on quality of life  \cite{Chen_chronic}, analyzing associations between exercise and morphology of an animal \cite{Zhang_mito}, and comparing distributions of agents' outcomes in reinforcement learning \cite{Park}. Moreover, the test of \eqref{eq:hypothesis} is frequently viewed as a non-parametric version of ANOVA \cite{ChenPokojovy,Rizzo} with myriad of scientific applications, typically comparing treatment outcomes between multiple groups. e.g. in clinical trials \cite{Cleophas} and cancer studies \cite{Heitjan, ZhangANOVA}. Table \ref{table:finite_sup_ex} outlines additional instances of scientific applications for $k$-sample inference when measures of interest have finite support, which is the case considered in this paper. 

This paper proposes an Optimal Transport approach to $k$-sample inference for $k\geq 2$ probability measures with finite supports in $\mathbb{R}^d$, $d \geq 1$. The method provides a powerful $k$-sample test of \eqref{eq:hypothesis}, but also allows comparison between different collections of $k$ measures in terms of their within-collection variability. Optimal Transport based approach has been shown successful in  one-sample (goodness-of-fit) and two-sample problems on finite \cite{Bigot,Klatt,Sommerfeld}, countable \cite{Tameling}, semidiscrete \cite{Hallin}, and some of the continuous spaces \cite{MunkCzado, Hundrieser}. The test statistics employ $p$-Wasserstein distances $W_p$ (or their regularized variants) to quantify differences between measures of interest while respecting metric structure of their supports \cite{Panaretos}. 

Our test statistic employs a different functional - the  Multimarginal Optimal Transport program $MOT$ \cite{Pass} - which can be represented as a variance functional on the space of measures \cite{CarlierMerigot} and thus serves as a natural candidate for testing variability in a collection of $k$ measures. We demonstrate that despite well-documented differences in solution structures of $MOT$ and $W_p$ problems (described, for example, in Sec. 1.7.4 of \cite{santambrogio}, or in \cite{Gerolin2019}), $MOT$ shares the same benefits as $W_p$ when it comes to the limiting behavior of its optimal value. 

Using $MOT$ for $k$-sample inference brings several important advantages. The main advantage is that the asymptotic distributions of $MOT$ can be derived under both $H_0$ and $H_a$. To the best of our knowledge, the only multivariate $k$-sample test statistic with known $H_a$ distribution is the one of \cite{Huskova}, where the limit laws are known only for a specific subset of alternatives. Our laws cover all alternatives in \eqref{eq:hypothesis}, which allows to explicitly derive a power function of the test and establish novel consistency results. The consistency analysis techniques developed in this paper can be further applied to one- and two-sample tests based on asymptotic results in \cite{Hundrieser, Sommerfeld,Tameling}.

Another benefit of $MOT$ limit under $H_a$ is the ability to estimate functionals of the $H_a$ distribution, e.g. Confidence Regions for the $MOT$ value. This allows for a novel application of comparing several collections of $k$ measures using overlap between their Confidence Regions (see Figure \ref{fig:year} for a concrete example). The procedure can be viewed as a distributional analogue of multiple comparisons (see \cite{Hsu} for a review), which are performed in a space of measures rather than Euclidean space. To the best of our knowledge, this type of analysis is not available with other $k$-sample statistics considered in the literature.

Conceptually, our approach to $k$-sample inference is equivalent to viewing $k$ measures as a collection of $k$ points and assessing variability within such collection. This approach lies in a general framework of Optimal Transport based distribution comparison: distributions are viewed as points in a Wasserstein space with Wasserstein distance indicating their closeness \cite{Panaretos}, and comparison is performed based on this distance (or its variants). This comparison framework has led to recent development of distributional analogues of traditional data analysis methods such as regression and time series \cite{ZhangTimeSeries, ZhuDist, Ghodrati}, synthetic controls \cite{GunsiliusDSC}, clustering \cite{TrillosJacobs}, and functional ANOVA \cite{Zemel_Gaus_proc}\footnote{Furthermore, 
	matching abilities of the solutions to resulting Optimal Transport problems - the multimarginal problems in particular -  have been of interest to economic theory (including recent results in \cite{Arieli2023} and \cite{Echenique2024}) and causal inference \cite{Gunsilius}.\\
}. Here, we view $k$-sample inference as a part of this framework.  

\par{\textbf{Organization of the paper:}} The rest of the Introduction formulates our approach to $k$-sample inference using $MOT$ (Sec. \ref{sec:formulation}), briefly outlines existing results on relevant statistical properties of Optimal Transport (Sec.  \ref{sec:weak_limits}) and summarizes our contributions (Sec. \ref{sec:contributions}). 

Section \ref{sec:main_results} provides our main results defining $k$-sample statistical inference procedures based on $MOT$ (Sec. \ref{sec:prelim} and \ref{sec:procedures}), reporting asymptotic distributions of the relevant test statistics (Sec. \ref{sec:asym}), and discussing consistency and power of the resulting procedures (Sec. \ref{subsec:power}). 

Section \ref{sec:sampling} shows how the procedures can be implemented. One can use certain forms of bootstrap to sample from the asymptotic distributions, which would result in solving large linear programs (Sec. \ref{sec:boot}), or one can approximate them with lower complexity linear programs (Sec. \ref{sec:null_approximation}). Alternatively, permutation approach can be implemented (Sec. \ref{sec:perm}). 

Section \ref{sec:applications} implements proposed procedures for testing $H_0$ and constructing Confidence Regions under $H_a$ (with $H_0$ and $H_a$ as defined by equation (\ref{eq:hypothesis}) above). Empirical power for testing $H_0$ and example Confidence Regions under $H_a$ are reported on the newly constructed as well as previously published synthetic data (Sec. \ref{sec:illustration_synth}). Applications to real data are reported next (Sec. \ref{sec:application_real}).  

Section \ref{sec:conclusion} summarizes the results of the paper (Sec. \ref{sec:summary}) and discusses limitations and future directions (Sec. \ref{sec:limitations}).

Appendix \ref{appA} contains proofs omitted in the main text. Appendix \ref{appB} contains additional technical results and detailed description of the data used in applications.

\subsection{Multimarginal Optimal Transport (MOT) for k-sample inference}\label{sec:formulation}
Let $\mu_1,\hdots,\mu_k$ be Borel probability measures supported on $\mathcal{X}\subseteq \mathbb{R}^d$, $d\geq 1$. The Multimarginal Optimal Transport ($MOT$) problem (equation (4.3) of \cite{Agueh}) is the optimization problem
\begin{equation}\label{eq:mot_general}
	\inf_{\pi \in \mathcal{C}(\mu_1,\hdots,\mu_k)} \int_{\mathcal{X} \times \cdots \times \mathcal{X}} c(x_1,\hdots,x_k) \, d\pi(x_1,\hdots,x_k) 
\end{equation}
where $\mathcal{C}(\mu_1,\hdots,\mu_k)$ is the set of Borel probability measures on the product space $\mathcal{X} \times \cdots \times \mathcal{X}$ with marginals $\mu_1,\hdots,\mu_k$. Different choices for the cost function $c(x_1,\hdots,x_k)$ are possible; throughout the paper, we fix the choice to be
\begin{equation}\label{eq:cost}
	c(x_1,\hdots,x_k):=\frac{1}{k} \sum_{i=1}^k \|x_i - \frac{1}{k}\sum_{j=1}^k x_j \|^2
\end{equation}
Under this choice, the $MOT$ problem is equivalent (Proposition 3.1.2 of \cite{Panaretos}) to the Wasserstein barycenter problem (equation 2.2 of \cite{Agueh})
\begin{equation}\label{eq:bary_general}
	\inf_{\nu \in \mathcal{P}^2(\mathbb{R}^d)} \frac{1}{k} \sum_{i=1}^k W_2^2(\mu_i,\nu)
\end{equation}
where $W_2(\cdot,\cdot)$ denotes 2-Wasserstein distance \footnote{Recent results on $1<p<\infty$-Wasserstein distance and Hellinger-Kantorovich distance cases include \cite{Brizzi2025} and \cite{Buze2025}, respectively.}.   
By equivalence here we mean that the optimal values of both programs are equal, and the optimal solutions $\pi^*$ of \eqref{eq:mot_general} and $\nu^*$ of \eqref{eq:bary_general} are related by $\nu^* = M_\# \pi^*$ where $M$ is the map that averages a given $k$-tuple of points from the supports of the $k$ measures ($M\# \pi^*$ stands for pushforward of a measure $\pi^*$ by the map $M$).  We remark here that when the measures are discrete, the barycenter problem (\ref{eq:bary_general}) generally has more than one optimal solution $\nu^*$. This presents challenges for statistical inference concerning barycenter \textit{solutions} \cite{LeGouic}, but does not impede the analysis of the optimal \textit{value} of barycenter or $MOT$ problems (recalling that all optimal solutions result in the same optimal value).

Let $MOT(\mu_1,\hdots,\mu_k)$ denote the optimal value of the $MOT$ program \eqref{eq:mot_general}. Observe that $MOT(\mu_1,\hdots,\mu_k) = 0$ if and only if the $k$ measures $\mu_1,\hdots,\mu_k$ are all the same. Indeed, if $\nu$ is (any) optimal solution to the barycenter problem (\ref{eq:bary_general}), then $MOT(\mu_1,\hdots,\mu_k)=0$ is equivalent to zero optimal value in the barycenter program \eqref{eq:bary_general}, i.e. $W_2^2(\mu_i,\nu) = 0$ for all $i=1,\hdots,k$. Due to metric properties of $2$-Wasserstein distance (Theorem 7.3 of \cite{Villani_topics}), this is equivalent to all the $k$ measures being the same (and equal to $\nu$). 

This observation suggests that testing $H_0$ in \eqref{eq:hypothesis} can be addressed via testing for $MOT(\mu_1$ $,\hdots,\mu_k)=0$. To this end, suppose that the data on $k$ samples $(X^1_i)_{i=1}^{n_1}, \cdots, (X^k_i)_{i=1}^{n_k}$  of sizes $n_1,\cdots,n_k$, respectively, is available to estimate the underlying measures $\mu_1, \cdots, \mu_k$ by the empirical measures $\hat{\mu}_1:=\frac{1}{n_i}\sum_{i=1}^{n_1} \delta_{x^1_i}, \cdots, \hat{\mu}_k:=\frac{1}{n_k}\sum_{i=1}^{n_k} \delta_{x^k_i}$. To test for $MOT(\mu_1,\hdots,\mu_k) = 0$ based on the data, we consider the asymptotic distribution $\mathcal{D}_0$ of the empirical estimator $MOT(\widehat{\mu}_1,\hdots,\widehat{\mu}_k)$ under $H_0$ and reject $H_0$ when the estimator value is large. 

More generally, once the asymptotic distribution $\mathcal{D}$ of empirical $MOT$ is known, one can estimate various functionals of $\mathcal{D}$, such as 
Confidence Regions (CR's) for a true $MOT$ value either $H_0$ or $H_a$ in \eqref{eq:hypothesis}. Inference of this type requires knowledge of the asymptotic distribution of empirical version of the $MOT$ value in \eqref{eq:mot_general}. Our derivation of these distributions leverages rich literature on asymptotic theory for the Wasserstein distance, whose main results we briefly review below.   

\begin{table}
	\caption{Examples of scientific data on finitely supported measures. Sample reference describes the set up, the data, and/or the comparison problem in each case.}
	\begin{tabular}{@{}lcl@{}}
		\hline
		Variable(s) of interest
		&  Support of measures & Ref. \\
		\hline
		Age of patients (in years)  & $\{0,1,2.\hdots,100 \} \subset \mathbb{R}$ & \cite{Desai}\\
		Tumor size (in mm) & $\{1,2, \hdots, 150 \} \subset \mathbb{R}$  & \cite{Sopik}\\
		Number of positive lymph nodes & $\{1,2, \hdots, 7 \} \subset \mathbb{R}$  & \cite{Fukui}\\
		Joint distributions of the above variables & finite subsets of $\mathbb{R}^2$ or $\mathbb{R}^3$ & \cite{Giordano}\\
		Cell counts & $\{0,1,\hdots,M\}^d \subset \mathbb{R}^d$ for $d$ sites & \cite{Corchado}\\
		Demand over $N$ locations & $N$ points (longitude, latitude) $ \in \mathbb{R}^2$ & \cite{Anderes}\\
		Disease rates over $N$ locations& $N$ points  on the map in $\mathbb{R}^2$& \cite{Khan}\\
		Pixel/voxel intensity in microscopy images  & the grid in $\mathbb{R}^2$ or $\mathbb{R}^3$ & \cite{Uchida}\\
		\hline
	\end{tabular}
	\label{table:finite_sup_ex}
\end{table}

\subsection{Existing results on weak limits for Optimal Transport}\label{sec:weak_limits}
The squared 2-Wasserstein distance $W_2^2(\mu,\nu)$ is the optimal value of the problem
\begin{equation}\label{eq:wass_dist}
	\inf_{\pi \in \mathcal{C}(\mu,\nu)} \int_{\mathcal{X} \times \mathcal{X}} c(x_1,x_2) \, d\pi(x_1,x_2) 
\end{equation}
with $c(x_1,x_2) = \|x_1 - x_2 \|^2$, which can be viewed as a particular case of the $MOT$ problem \eqref{eq:mot_general} with $k=2$ measures. Being a true metric on a space of probability measures on a given metric space (\cite{Villani_oldnew}), the 2-Wasserstein distance $W_2$ (and, more generally, the $p$-Wasserstein distance $W_p$) provides a natural way to compare probability measures while respecting the geometry of the supporting metric space. Under this framework, the true measures are estimated by their empirical counterparts, and statistical inference is conducted using limiting laws for the empirical Wasserstein distance \cite{Panaretos_Review, Ramdas}. 

The forms of the weak limits depend on two main factors: dimensionality of the support and the nature of the measures (where the cases $\mu=\nu$ and $\mu \neq  \nu$ may have different limits). 
Letting $OT_c$ denote the optimal value in \eqref{eq:wass_dist} (with possibly different costs $c$), the limiting laws have general form of 
\begin{equation}\label{eq:weak_lim}
	\rho_n\left(OT_c(\widehat{\mu}_{n_1}, \widehat{\nu}_{n_2}) - OT_c(\mu,\nu)   \right) \xrightarrow[]{\text{in law}} L
\end{equation}
with $\rho_n=\sqrt{\frac{n_1n_2}{n_1+n_2}}$ (when only $\mu$ is estimated from the data while $\nu$ is not, the ``one-sample" version with $\rho_n=\sqrt{n}$ is considered). When measures are supported on $\mathbb{R}$, the limits $L$ can be Gaussian, with variance that depending on the truth $\mu,\nu$ under the ``alternative" assumption $\mu\neq\nu$ \cite{MunkCzado, delBarrio_alt}, and are non-Gaussian under the ``null" assumption $\mu=\nu$ \cite{delBarrio_alt,delBarrio_L2}. When measures are supported on $\mathbb{R}^d$, $d>1$, and are absolutely continuous, the curse of dimensionality takes place: the empirical Wasserstein distance converges in expectation to the true one too slowly \cite{Dudley_glivenkocantelli, Fournier}. It is still possible, however, to obtain convergence statement similar to \eqref{eq:weak_lim} in any dimension $d\geq 1$ by replacing the centering true value $OT_c(\mu,\nu)$ with expectation of the empirical value $\mathbb{E}\left(OT_c(\widehat{\mu}_{n_1}, \widehat{\nu}_{n_2}\right)$ \cite{delBarrioLoubes}. The limit $L$ is Gaussian when $\mu \neq \nu$ and is degenerate (i.e. limiting random variable has zero variance) when $\mu=\nu$\footnote{For recent review of results on centering with true value versus centering with expectation and the effect of regularization on the associated limits, see \cite{WeedCLT}.}.

Favorable situation arises when measures are supported on a finite space $\mathcal{X}=\{x_1,\hdots,x_N \}$ $\subseteq \mathbb{R}^d$: \cite{Sommerfeld} show that the limit law of the form \eqref{eq:weak_lim} hold for the $W_p$ distance in any dimension $d \geq 1$ and use resulting laws to construct statistical inference under $H_0$ and $H_a$ (the case of countable support is treated in \cite{Tameling}). The laws under either $H_0$ or $H_a$ are non-degenerate and given by 
\begin{equation}\label{eq:munk_limit}
	\rho_n \left( W_p^p( \widehat{\mu}_{n_1}, \widehat{\nu}_{n_2}) - W_p^p(\mu,\nu)\right) \xrightarrow[]{\text{in law}} \max_{(u_1,u_2)\in\Phi^*}\sqrt{\lambda}\langle  u_1, G_1 \rangle + \sqrt{1-\lambda}\langle  u_2, G_2 \rangle
\end{equation}
where $G_1,G_2$ are the weak limits of multinomial processes $\sqrt{n_1}(\widehat{\mu}_{n_1} - \mu)$ and $\sqrt{n_2}(\widehat{\nu}_{n_2} - \nu)$,
and $\Phi^*$ is set of optimal solutions to the dual of the $W_p^p$ program \eqref{eq:wass_dist}. The results are extended in \cite{Hundrieser} to general measures supported on $\mathbb{R}^d$, $d=1,2,3$ and general costs $c$ (with discussion of limitations in higher dimensions $d \geq 4$), thus providing a unified approach to weak limits of empirical OT costs centered by the true population value.     

The starting point for theoretical results of this paper is the weak limit \eqref{eq:munk_limit} \cite{Sommerfeld}. Inspired by this result, we establish the limits of the form \eqref{eq:weak_lim}, where $OT_c$ is now the optimal value of the $MOT$ program \eqref{eq:mot_general} with $k \geq 2$ measures supported on a finite space $\mathcal{X}=\{x_1,\hdots,x_N \} \subseteq \mathbb{R}^d$ for any $d \geq 1$. The implications of these results to $k$-sample inference and further theoretical findings related to our limits are summarized below. 

\subsection{Summary of  contributions and outline}\label{sec:contributions} 
\par \textit{Asymptotic distributions of $MOT$:} We provide asymptotic distribution of the optimal value of $MOT$ on finite spaces by establishing Hadamard directional differentiability of the $MOT$ functional and combining it with functional Delta method \cite{Carcamo,Fang,Hundrieser,Shapiro_asymptotic,Shapiro,Romisch, Sommerfeld,Tameling}. The resulting limit is a Hadamard directional derivative of $MOT$ at the true $\mu:=\left(\mu_1,\hdots,\mu_k \right)$ in the direction of the limit $G$ of the empirical process $\rho_n \left( \widehat{\mu}_n - \mu \right)$ for suitably defined rate $\rho_n$ (Theorem \ref{thm:as_dist}(a)). For $k=2$ measures, our limit recovers the one for the Wasserstein distance on finite spaces obtained in \cite{Sommerfeld}; for $k>2$ measures, our limit allows to construct novel inference procedures for the $k$-sample problem using $MOT$ (Section \ref{sec:procedures}). 

We specify the structure of the limit under the assumption of $H_0$ and construct a low complexity stochastic upper bound on the null distribution (Theorem \ref{thm:as_dist}(b)) that is used to efficiently approximate the limit under $H_0$ (Section \ref{sec:null_approximation}). The bound is tight for $k=2$. 

We further specify the structure of the limit under $H_a$ and provide sufficient conditions for the limit to be Gaussian by leveraging the results from geometry of multitransportation polytopes from \cite{Emelichev}.  When the limits are not Gaussian, we construct the Normal lower bounds on the alternative limiting distribution (Theorem \ref{thm:as_dist}(c)). Our stochastic bounds on the null and alternative distributions provide an analytically tractable way to study the power of the Optimal Transport based tests (Section \ref{subsec:power}), which, to the best of our knowledge, was not yet considered in the literature.

\par \textit{Consistency and power:} 
We provide a novel power analysis for Optimal Transport based tests of hypotheses \eqref{eq:hypothesis} that encompasses both our test \eqref{eq:test} and the two-sample test in \cite{Sommerfeld} and can potentially be applied to tests based on limiting laws in \cite{Tameling} and \cite{Hundrieser}. We show consistency of the test under fixed alternatives (Proposition \ref{prop:consistency_fixed}), as well as uniform consistency in a certain broad class of alternatives (Theorem \ref{thm:unif_consistency_F} and Proposition \ref{prop:unif_consist_FK}). 

We illustrate theoretical power results in $k=2$ case by providing a lower bound on the power function that explicitly relates sample size and the effect size (Corollary \ref{coro:power}, Figure \ref{fig:power_real}). We also quantify how the population version of our statistic changes with number of measures for certain sequences of alternatives (Lemmas \ref{lem:clustered_alt} and \ref{lem:sparse_alt}), suggesting potential power advantages in these cases. For the case of small sample sizes, we provide a permutation version of the $MOT$ based test (Section \ref{sec:perm}). Comparison with with state-of-the-art tests of \cite{Huskova,Kim,Rizzo} shows strong finite sample power performance of our tests (Figure \ref{fig:synth_3D}).

\par \textit{Computational complexity:} Leveraging recent complexity results  \cite{Altschuler} for $MOT$ (or barycenter) program, we prove polynomial time complexity of the derivative bootstrap that consistently estimates asymptotic distribution of $MOT$ under $H_0$ (Lemma \ref{lem:der_boot_compl}). Polynomial complexity of m-out-of-n bootstrap and permutation procedure follow directly from \cite{Altschuler} (Table \ref{table:complexity}). We also demonstrate that, in addition to bootstrap sampling, the null upper bound of Theorem \ref{thm:as_dist}(b) can efficiently approximate the null distribution when the cardinality of the support $N=|\mathcal{X}|$ is large (Figure \ref{fig:ub}). 

\par \textit{Applications to real data and software:} We illustrate performance of $MOT$ based $k$-sample inference on two synthetic datasets showing strong power performance when testing $H_0$ and the ability to produce meaningful and interpretable confidence regions under $H_a$ (Section \ref{sec:illustration_synth}). Further, we apply our methodology to real data on cancers in the United States populations to confirm claims in cancer studies that were previously made using different methodologies (Section \ref{sec:application_real}, Figures \ref{fig:sizes_1D} and \ref{fig:year}). Current version of the software that implements our methods is available at \url{https://github.com/kravtsova2/mot}.

\section{k-Sample inference on finite spaces using 
	$MOT$}\label{sec:main_results}

\subsection{Notation and preliminary definitions}\label{sec:prelim}
Denote the vector of $k$ true measures supported on $\mathcal{X}=\{x_1\hdots,x_N\} \subseteq \mathbb{R}^d$ by  
$$\mu:=\begin{pmatrix} \mu_1 \\ \vdots\\ \mu_k \end{pmatrix} \in \mathbb{R}^{kN}$$
and the vector of the empirical counterparts $\widehat{\mu}_i:=\frac{1}{n_i}\sum_{j=1}^{n_i}\delta_{x_j}$ by
$$\widehat{\mu}_n:=\begin{pmatrix} \widehat{\mu}_1 \\ \vdots\\ \widehat{\mu}_k  \end{pmatrix} \in \mathbb{R}^{kN}$$
with sample sizes
$$n:=(n_1,\hdots,n_k)$$
where $n\to \infty$ to be interpreted as each sample size tending to infinity. 

Let $MOT(\mu)$ be the optimal value of the program \eqref{eq:mot_general}, which on the finite space $\mathcal{X}$ becomes 
the finite-dimensional linear program 
\begin{equation} \label{eq:mot_primal}
	\begin{aligned}
		& \min_{\pi \geq 0} \langle c,\pi \rangle \\
		& A\pi = \mu
	\end{aligned}
\end{equation}
The optimization variable $\pi \in \mathbb{R}^{N^k}$ is a column vector representing joint probability distribution with marginals $\mu_1,\hdots,\mu_k$ (frequently called \textit{multicoupling}), a matrix $A \in \mathbb{R}^{kN \times N^k}$ encodes the constraints for $\pi$ to be a multicoupling (i.e. that summing certain entries of $\pi$ gives the marginals $\mu_1,\hdots,\mu_k$), and a cost column vector $c \in \mathbb{R}^{N^k}$ contains Euclidean distances between measure support points to their averages given by \eqref{eq:cost}\footnote{The linear program formulation of MOT problem is discussed, for example, on p.3 of \cite{Lin}.}. 

For reader's convenience, Example \ref{ex:main_example} below illustrates the structure of the linear program (\ref{eq:mot_primal}) in a case of three measures, each supported on two points in $\mathbb{R}^1$: 

\begin{example}[\textbf{Illustration of $MOT$ optimization problem for three measures}] \label{ex:main_example} Consider the finite set $\mathcal{X}=\{5,10\} \subset \mathbb{R}^1$ which could represent, for instance, tumor sizes (in centimeters) of cancer patients, and consider probability measures supported on $\mathcal{X}$ representing the probabilities of occurrences of $5 \text{cm}$ and $10 \text{cm}$ tumors in given population of cancer patients. Suppose that the measure $\mu_1$ has probabilities recorded in a vector $\mu_1:=\begin{pmatrix} {\mu_1}_1 \\ {\mu_1}_2 \end{pmatrix}$, and, similarly, $\mu_2:=\begin{pmatrix} {\mu_2}_1 \\ {\mu_2}_2 \end{pmatrix}$, $\mu_3:=\begin{pmatrix} {\mu_3}_1 \\ {\mu_3}_2 \end{pmatrix}$. The multimarginal optimal transport problem (\ref{eq:mot_primal}) is to minimize a linear (with coefficients in $c$) function of a measure $\pi$ on the product space $\mathcal{X} \times \mathcal{X} \times \mathcal{X}$ whose marginals are $\mu_1, \mu_2$, and $\mu_3$, respectively (in this example, all of these sets are equal to $\mathcal{X}$). Technically, $\pi$ is an order-$3$ tensor, i.e. an array with $3$ indices with values in $\{1,2\}$, but for notational convenience we represent it by a long vector  $(\pi_{ijk})_{i,j,k \in \{1,2 \}} \in \mathbb{R}^{2\cdot 2 \cdot 2}$. The cost $c_{ijk}$ in the objective of (\ref{eq:mot_primal}) associated with the entry $\pi_{ijk}$ is the average of squared differences (or squared norms of the differences in higher-dimensional case) between support points $x_i, x_j, x_k$ to their mean $\overline{M}_{ijk}=\frac{1}{3}(x_i+x_j+x_k)$, i.e.
	$$c_{ijk} = \frac{1}{3} \left((x_i - \overline{M}_{ijk})^2 + (x_j - \overline{M}_{ijk})^2  + (x_k - \overline{M}_{ijk})^2  \right)$$
	The objective is to minimize the total discrepancy weighted by $\pi$, which is given by 
	\begin{equation*}
		\langle c, \pi \rangle = c_{111}\pi_{111} + c_{112}\pi_{112} + \cdots + c_{222}\pi_{222} 
	\end{equation*}
	The multicoupling $\pi$ is subjected to having non-negative entries and constrained linearly with $A\pi = \mu$. The constraint matrix $A$ is responsible for making sure that appropriate entries of $\pi$ sum to the given marginals $\mu_1$, $\mu_2$, and $\mu_3$, i.e.
	\begin{equation} \label{eq:ex}
		\underbrace{\begin{pmatrix} 1&1&1&1&0&0&0&0 \\ 0&0&0&0&1&1&1&1 \\ 
				1&1&0&0&1&1&0&0 \\ 0&0&1&1&0&0&1&1 \\ 1&0&1&0&1&0&1&0 \\ 0&1&0&1&0&1&0&1  \end{pmatrix}}_{A}  \underbrace{\begin{pmatrix} \pi_{111} \\ \pi_{112} \\ \pi_{121} \\ \pi_{122} \\ \pi_{211} \\ \pi_{212} \\ \pi_{221} \\ \pi_{222}\end{pmatrix}}_{\pi} = \underbrace{\begin{pmatrix} {\mu_1}_1 \\ {\mu_1}_2 \\ {\mu_2}_1 \\ {\mu_2}_2 \\ {\mu_3}_1 \\ {\mu_3}_2  \end{pmatrix}}_{\mu}
	\end{equation}
	finishing the example.
\end{example}

\vspace{1 cm}

The dual program of (\ref{eq:mot_primal}) is given by
\begin{equation} \label{eq:mot_dual}
	\begin{aligned}
		& \max_{u} \langle u, \mu \rangle \\
		& A'u \leq c
	\end{aligned}
\end{equation}
(the derivation of the dual follows from the standard theory of linear programming, e.g. Section 4.1 of \cite{Bertsimas}). A column vector $u:=\begin{pmatrix} u_1 \\ \vdots \\ u_k \end{pmatrix} \in \mathbb{R}^{kN}$ contains dual variables, one for each measure, and the objective of (\ref{eq:mot_dual}) can be thought of summing the contributions $\langle u_1, \mu_1 \rangle + \cdots + \langle u_k, \mu_k \rangle$. 

Let $\Phi^*$ denote the set of dual optimal solutions to \eqref{eq:mot_dual}. This set consists of all vectors $u$ that result in the maximum value of the dual objective (= minimum value of the primal objective by strong duality, e.g. Theorem 4.4 of \cite{Bertsimas}) and satisfy the dual constraints, i.e.

\begin{equation}\label{eq:dualsol}
	\Phi^*:=\{u=\begin{pmatrix}u_1 \\ \vdots \\ u_k \end{pmatrix} \in \mathbb{R}^{kN}: \langle u, \mu \rangle = MOT(\mu), \text{ } A'u \leq c \}
\end{equation}

We consider the asymptotic behavior of scaled and centered empirical estimator $MOT(\widehat{\mu}_n)$ by establishing the weak limit
$$\rho_n\left(MOT(\widehat{\mu}_n)-MOT(\mu)\right) \overset{\text{ in law }}{\longrightarrow} X$$
as $n \to \infty$ where $X\overset{d}{=} X_0$ under $H_0$ and $X\overset{d}{=} X_a$ under $H_a$. The set $\Phi^*$ will be needed to define the limit $X$. 

\subsection{Definitions of $H_0$ testing and $H_a$ inference procedures}\label{sec:procedures}
Consider the statistic 
$$T_n:=\rho_n\left(MOT(\widehat{\mu}_n)-MOT(\mu)\right)$$
where $MOT(\mu)=0$ under $H_0$ and $MOT(\mu)>0$ under $H_a$.

An $\alpha$-level test of $H_0$ would reject $H_0$ if $\rho_n MOT(\widehat{\mu}_n)$ is large, i.e. if $T_n$ exceeds a $(1-\alpha)$th quantile of its null distribution $\mathcal{D}_0$. However, as Theorem \ref{thm:as_dist} shows, $\mathcal{D}_0$ depends on the unknown true $\mu$, and hence care must be taken to ensure that the estimated cut-off used for the test still results in the (asymptotic) level $\alpha$. 

To this end, we consider a consistent bootstrap estimator of $\mathcal{D}_0$ given in Proposition \ref{prop:boot} and denote its $(1-\alpha)$-th quantile by $c_{\alpha,\mathcal{D}_0}$. Consistency of the bootstrap is shown using results of \cite{Fang}, and by Corollary 3.2 of the same work such bootstrap based cut-off gives an asymptotic level $\alpha$ test of $H_0$. Using this cut-off, we define the asymptotic test of $H_0$ as a map
\begin{equation}\label{eq:test}
	\phi_{n,\mu}:=
	\begin{cases}
		1 & \text{ if } T_n \geq c_{\alpha,\mathcal{D}_0} \\
		0 & \text{ otherwise }
	\end{cases}
\end{equation}

Similarly, the distribution of $T_n$ under $H_a$ is consistency estimated by bootstrap in Proposition \ref{prop:boot}, with resulting $(\alpha/2)$-th and $(1-\alpha/2)$-th quantiles denoted by $c_{\alpha/2,\mathcal{D}_a}$ and $c_{1-\alpha/2,\mathcal{D}_a}$, respectively. The asymptotic $(1-\alpha)\%$ Confidence Region for $MOT(\mu)$ under $H_a$ is given by
\begin{equation}\label{eq:CR}
	\left(\frac{1}{\rho_n} MOT(\widehat{\mu}_n) - c_{1-\alpha/2,\mathcal{D}_a}, \frac{1}{\rho_n} MOT(\widehat{\mu}_n) - c_{\alpha/2,\mathcal{D}_a}  \right)
\end{equation}

\subsection{Asymptotic distributions of $MOT$ under $H_0$ and $H_a$}\label{sec:asym}

\begin{figure}
	\includegraphics[width=1\textwidth]{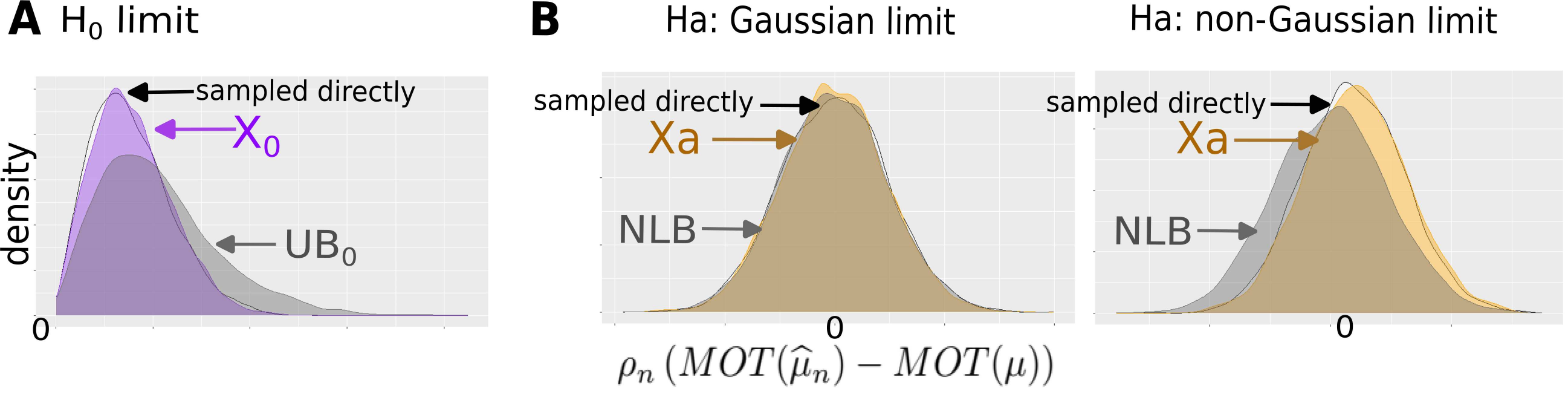}
	\caption{Illustration of asymptotic distributions of $MOT$ given by Theorem \ref{thm:as_dist} on the set up of Example \ref{ex:main_example}. A. Under $H_0$ (Theorem \ref{thm:as_dist}(b)): null distribution $X_0$, the upper bound $UB_0$, and $MOT$ values sampled directly (black line) by sampling empirical measures from the truth and evaluating $MOT$ value. All densities here and in the rest of the paper are estimated by kernel density estimators in ggplot2 \cite{ggplot2} with default parameters. B. Under $H_a$ (Theorem \ref{thm:as_dist}(b)): Gaussian limit under Condition (A1) (Left) and non-Gaussian limit (Right) with Normal Lower Bounds (NLB's) ($NLB\overset{d}{=} X_a$ in the Gaussian case). }
	\label{fig:limit_laws}
\end{figure}

Theorem \ref{thm:as_dist}(a) provides the general form of the asymptotic distribution of $MOT(\widehat{\mu}_n)$ on finite spaces. This distribution is given by the Hadamard directional derivative of the $MOT$ functional, which is an optimal value of a linear program with a feasible set consisting of dual optimal solutions $\Phi^*$  \eqref{eq:dualsol}. If $\Phi^*$ is a singleton, the limit in Theorem \ref{thm:as_dist}(a) is a linear combination of Gaussians, and hence is also Gaussian. If not, the limit is the maximum (taken over the feasible set $\Phi^*$) of such linear combinations\footnote{Note that adding a constant to any dual vector $u_{i}$ and subtracting the same constant from any other dual vector $u_{i'}$ (or distributing such constant over several other dual vectors and subtracting) does not change the dual objective  and does not violate the dual constraints in \eqref{eq:mot_dual}. So by uniqueness of the dual variables $u=(u_1,\cdots,u_k)$ we actually mean uniqueness up to this operation, which can be viewed as considering equivalence classes of dual solutions under this operation.}.

By the theory of linear programming, it is possible to assess whether the set of dual optimal solutions $\Phi^*$ is a singleton or not based on the corresponding set of basic optimal solutions to the primal program (we use Theorem 5.6.1 of \cite{Sierksma}, and Chapters 4 and 5 of \cite{Bertsimas} for the general linear programming results employed below). In our case, the basic optimal solutions to the primal program \eqref{eq:mot_primal} are the vertices of a multitransportation polytope $P(\mu_1,\hdots,\mu_k):=\{\pi \geq 0: A\pi = \mu\}$ given by multicouplings $\pi^*$. These vertices contain $\leq \text{rank}(A) = kN - k +1$ positive entries, and a vertex is termed \textit{degenerate} if it contains strictly less (p. 366 of \cite{Emelichev}). 

The dual optimal set $\Phi^*$ cannot be a singleton if an optimal solution to the primal $MOT$ program is unique and degenerate. This is always the case under $H_0$: the unique optimal solution $\pi^*$ is given by the ``identity" multicoupling (with $\mu_1$ in the entries with the same tuple indices and zeros otherwise - so it is degenerate). Hence, the asymptotic distribution of $MOT$ under $H_0$ is never a Gaussian (Theorem \ref{thm:as_dist}(b)).

The dual optimal set $\Phi^*$ is a singleton if there exists a \textit{non-degenerate} primal optimal vertex, i.e. an optimal multicoupling $\pi^*$ with $kN-k+1$ positive entries. This can be possible under certain $H_a$'s; in particular, this is possible if  a multitransportation polytope $P(\mu_1,\hdots,\mu_k):=\{\pi \geq 0: A\pi = \mu\}$ contains no degenerate vertices. In this case, $\Phi^*$ is a singleton, and the corresponding asymptotic distribution of $MOT$ is Gaussian (Theorem \ref{thm:as_dist}(c)). We use a result in discrete geometry from \cite{Emelichev} to provide a sufficient condition (A1) that leads to Gaussian limits under $H_a$: 

\begin{condition}[A1]\label{assn:non_degeneracy} (\textit{Regularity}, Definition 1.5 in Chapter 8 of \cite{Emelichev})
	For each $i=1,\hdots,k$, order the entries of the vector $\mu_i$ 
	as $\mu_i^1 \geq \mu_i^2 \geq \cdots \geq \mu_i^N$, and assume that $\mu_i^1 < \mu_{i+1}^1$ for $i=1,\hdots,k-1$. A multitransportation polytope $P(\mu_1,\hdots,\mu_k)$ is \textit{regular} if 
	$$\mu_i^N + \sum_{j=i+1}^k \mu_j^1 > k-i, \hspace{0.5 cm} \forall  i=1,\hdots,k-1$$
	A regular multitransportation polytope does not have degenerate vertices (Lemma 1.4 in Chapter 8 of \cite{Emelichev}).
\end{condition}

\begin{remark}\label{rmk:non_degeneracy_k2} For $k=2$ measures, the condition (A1) is implied by the following \textit{No Subset Sum} condition that ensures that a transportation polytope $P(\mu_1,\mu_2)$ has no degenerate vertices\footnote{This condition is mentioned by \cite{Staudt} for uniqueness of Kantorovich potentials for $k=2$ finitely supported measures.}: there is no proper subsets of indices $I,J \subset [N]$ such that $\sum_{i \in I}\mu_1^i = \sum_{j \in J} \mu_2^j$. This condition is both necessary and sufficient to exclude degenerate vertices in $k=2$ case (see, e.g., Theorem 1.2 in Chapter 6 of \cite{Emelichev}). 
\end{remark}

\begin{theorem}[\textbf{Asymptotic distribution of $MOT$ on finite spaces}]\label{thm:as_dist} Assume that the sizes of $k$ samples $n_1, \hdots, n_k \to \infty$ satisfying $\frac{n_i}{n_1 + \hdots + n_k} \to \lambda_i \in (0,1)$. Denote $\rho_n := \frac{\sqrt{n_1\cdot \hdots \cdot n_k}}{{\left(\sqrt{n_1+ \hdots + n_k}\right)^{k-1}}}$ and $a_i := \prod_{j \neq i} \lambda_j$. Then,
	\begin{itemize}
		\item[(a)] The asymptotic distribution of $MOT$ is given by
		\begin{equation}\label{eq:limdist}\rho_n\left(MOT(\widehat{\mu}_n) - MOT(\mu)   \right) \xrightarrow[]{\text{in law}} \max_{u \in \Phi^*} \sum_{i=1}^k \sqrt{a_i}\langle u_i, G_i   \rangle
		\end{equation}
		where the feasible set $\Phi^*$ is given by \eqref{eq:dualsol}, and $G_i \overset{\text{indep.}}{\sim} N(0,\Sigma_i)$ with $\Sigma_i = \text{diag}(\mu_i) - \mu_i \mu_i'$. 
		\item[(b)] Under $H_0$, the limit in \eqref{eq:limdist} is non-Gaussian, and given by
		\begin{equation*}\label{eq:limdist_null}
			\rho_n\left(MOT(\widehat{\mu}_n) - 0   \right) \xrightarrow[]{\text{in law}} X_0 \sim \mathcal{D}_0
		\end{equation*}
		where $\mathcal{D}_0$ is given by
		\begin{equation}\label{eq:limdist_null_expression}
			\begin{aligned}
				\max_{u} & \sum_{i=1}^k \sqrt{a_i}\langle u_i, G_i   \rangle \\
				\text{ s.t. } & \sum_{i=1}^k u_i = 0 \\
				& A'u \leq c
			\end{aligned}
		\end{equation}
		with $G_i \overset{\text{indep.}}{\sim} N(0,\Sigma_1)$ with $\Sigma_1 = \text{diag}(\mu_1) - \mu_1 \mu_1'$.
		
		Furthermore, there exists ${UB}_0 \sim \mathcal{D}_{{UB}_0}$ on the same probability space as $X_0$ such that ${UB}_0 \geq X_0$ everywhere, and $\mathcal{D}_{{UB}_0}$ is given by
		\begin{equation}\label{eq:UB_0}
			\begin{aligned}
				\max_{u} & \sum_{i=2}^k \langle u_i, \sqrt{a_i}G_i - \sqrt{a_1}G_1   \rangle \\
				& \widetilde{A}'u \leq \widetilde{c}
			\end{aligned}
		\end{equation}
		where $\widetilde{A}' u \leq \widetilde{c}$ is a subset of constraints from \eqref{eq:limdist_null_expression}.
		\item[(c)] Under $H_a$, 
		\begin{equation*}\rho_n\left(MOT(\widehat{\mu}_n) - MOT(\mu)   \right) \xrightarrow[]{\text{in law}} X_a \sim \mathcal{D}_a
		\end{equation*}
		where $\mathcal{D}_a$ is given by \eqref{eq:limdist}.
		Furthermore, for every $u^* \in \Phi^*$ given by \eqref{eq:dualsol}, there exists a random variable $NLB_{u^*}$ on the same probability space as $X_a$, such that $NLB_{u^*} \leq X_a$ everywhere, and 
		\begin{equation}\label{eq:limdist_alt_ub}
			NLB_{u^*}\sim \mathcal{N}\left(0,\sum_{i=1}^k a_i{u_i^*}'\Sigma_i u_i^*\right)
		\end{equation}
		
		If Condition (A1) holds, then $\Phi^*$ is a singleton $\{u^* \}$, and $X_a \overset{d}{=}NLB_{u^*}$. 
	\end{itemize}
\end{theorem}
\begin{proof}[Proof summary for Theorem \ref{thm:as_dist}]
	Proof of part (a) is outlined below, with details in Appendix \ref{app:proof_thm_a}. In what follows, for each $i=1,\hdots,k$, we view the measures $\mu_i \in \mathcal{P}(\mathcal{X})\subseteq\left( l^1(\mathcal{X}), \| \cdot\|_{l^1} \right)$, and the dual vectors $u_i \in \left(l^\infty(\mathcal{X}), \| \cdot\|_{l^\infty}\right)$. The weak convergence and Hadamard directional differentiability are with respect to $l^1$ norm on $\bigotimes_{i=1}^k l^1(\mathcal{X})$. 
	\begin{itemize}
		\item[] \underline{Step 1} Establish, for a suitable scaling $\rho_n$, the weak limit $\sqrt{a}G$ of the empirical process
		$$\rho_n\left( \widehat{\mu}_n - \mu  \right) \xrightarrow[]{\text{in law}} \sqrt{a} G$$
		
		\item[] \underline{Step 2} Confirm that the functional $\mu \longrightarrow MOT(\mu)$ is Hadamard directionally differentiable at $\mu$ with derivative
		$$f_\mu'(G) = \max_{u \in \Phi^*} \sum_{i=1}^k \sqrt{a_i}\langle u_i, G_i   \rangle$$
		\item[] \underline{Step 3} Use Delta Method for Hadamard directionally differentiable maps \cite{Romisch, Shapiro_asymptotic} to conclude that
		$$\rho_n \left(f(\widehat{\mu}_n) - f(\mu) \right) \xrightarrow[]{\text{in law}} f'_\mu(G)$$
	\end{itemize}
	
	Proof of part (b) is given in Appendix \ref{app:proof_thm_b}. It provides the exact form of the proposed upper bound $UB_0$ reporting the constraints in $\widetilde{A}$. To construct $\widetilde{A}$, we consider how the inequality constraints $A'u\leq c$ behave on the kernel of the map $\{u \longrightarrow \sum_{i=1}^k u_i\}$. Resulting upper bound has only polynomially many constraints and can be sampled efficiently to approximate the null distribution in Section \ref{sec:null_approximation}. We remark that the proposed bound is not unique: in particular, it can be strengthened by including more constraints from $A$ (Section \ref{sec:limitations}). The proposed bound is tight when $k=2$\footnote{We remark here that the null distribution program \eqref{eq:limdist_null_expression} can be written with $k-1$ dual variables instead of $k$ due to the constraint $\sum_{i=1}^k ui = 0$; this is what \cite{Sommerfeld} refer to in $k=2$ case (p. 227). We choose to keep the form \eqref{eq:limdist_null_expression} for notational convenience}.  
	
	For part (c), if the limit is not Gaussian (i.e., the feasible set $\Phi^*$ is not a singleton), one can take any $u^* \in \Phi^*$ and consider the (random) objective \eqref{eq:limdist} evaluated at $u^*$. Resulting value lower bounds the value of the maximization program \eqref{eq:limdist}, and it is distributed according to \eqref{eq:limdist_alt_ub}.  
	
\end{proof}

\begin{obs}\label{obs:cost_dist}
	The entries of the cost vector $c_{i_1i_2\cdots i_k}$ indexed by $i_1,\cdots,i_k \in \{1,\hdots,N \}$ with $k-1$ coinciding index values can be written in terms of the distance between two points with unique indices scaled by $\frac{k-1}{k^2}$. For example, $c_{1\cdots1 2} = \frac{k-1}{k^2} \| x_1 - x_2\|^2$. The details are provided in Appendix \ref{app:proof_cost_dist}.
\end{obs}
\begin{lemma}[\textbf{Bounds on the dual variables}]\label{lem:dual_bd_null}
	Fix $k\geq 2$. Let $u = (u_1,\hdots,u_k)$ be optimal solutions to the dual $MOT$ program \eqref{eq:mot_dual} satisfying $\sum_{i=1}^k u_i = 0$ and chosen\footnote{Recall that adding a constant to any dual variable $u_{i}$ and subtracting the same constant from any other dual variable $u_{i'}$ does not change the dual objective  and does not violate the dual constraints in \eqref{eq:mot_dual}. Hence, given any vector of dual solutions $u=(u_1,\hdots,u_k)$, the first entries of $u_2,\hdots,u_k$ can be normalized to zero, and the constraint $\sum_{i=1}^k u_i = 0$ would force the first entry of $u_1$ to be zero as well. Such normalization is frequently done to avoid redundant solutions - see, e.g., the definition of dual transportation polyhedron in \cite{Balinski}.} such that the first entries ${u_i}_1 = 0$. Then for each $i=1,\hdots,k$, the $j$th entry of  $u_i$ is bounded as
	$$|u_i|_j \leq \frac{k-1}{k^2}  \|x_1 - x_j \|^2$$
	where $\|\cdot-\cdot \|^2$ is the squared distance on the ground metric space $\mathcal{X}=\{x_1,\hdots, x_N \}$. It follows that
	$$\| u_i\| \leq \frac{k-1}{k^2} C(\mathcal{X}), \text{ } i=1, \hdots, k$$
	where $C(\mathcal{X})$ depends only on the ground metric space $\mathcal{X}$.
\end{lemma}

\begin{remark}\label{rmk:slackness}
	The assumption $\sum_{i=1}^k u_i = 0$ holds for those $\mu$ for which
	the primal optimal solution - the multicoupling $(\pi_{i_1,\hdots,i_k})_{i_1,\hdots,i_k}$ - assigns a positive mass to the ``diagonal" tuples $(x_{i_1},\hdots,x_{i_1})$, i.e. $\pi_{i_1,\hdots,i_1}>0$ for $i_1=1,\hdots,N$. By complementary slackness result in linear programming (see, e.g., Theorem 4.5 of \cite{Bertsimas}), in this case, the corresponding constraints of the dual 
	$$(u_1)_j + \cdots + (u_k)_j \leq c_{j,\hdots,j} = 0, \text{ } j=1,\hdots,N$$
	hold with equality, giving $\sum_{i=1}^k u_i = 0$. This always holds under $H_0$ and frequently happens under $H_a$. 
\end{remark}

Using the above results, Proposition \ref{prop:c_alpha_bd} defines an upper bound on all test cut-offs $c_{\alpha,\mathcal{D}_0}$, which is independent of the nature and the number of measures in $\mu(k)=(\mu_1,\hdots,\mu_k)$. This bound is used to prove consistency of the test \eqref{eq:test} with cut-offs $c_{\alpha,\mathcal{D}_0}$ uniformly over $\mu$ (Theorem \ref{thm:unif_consistency_F}) and $k$ (Proposition \ref{prop:unif_consist_FK}).  

\begin{prop}[\textbf{Bound for test cut-off $c_{\alpha,\mathcal{D}_0}$}]\label{prop:c_alpha_bd}
	Fix the test level $\alpha \in (0,1)$. There exists $c_\alpha(\mathcal{X})$ depending only on the ground metric space $\mathcal{X}$ such that
	$$c_{\alpha,\mathcal{D}_0} \leq c_{\alpha}(\mathcal{X}) \text{ for all } \mu(k) \text{ supported on }\mathcal{X}$$
	where $\mu(k)=(\mu_1,\hdots,\mu_k)$. In particular, for any $k\geq 2$, 
	$$c_{\alpha,\mathcal{D}_0} \leq \left(\sum_{i=1}^k \sqrt{a_i}\right) \frac{k-1}{k^2}C(\mathcal{X})\sqrt{-8\ln{(\alpha/4)}}=: c_\alpha(\mathcal{X})$$
	For equal sample sizes $n_1=\cdots=n_k$, this gives
	$$c_{\alpha,\mathcal{D}_0} \leq  \frac{k-1}{k^{(k+1)/2}}C(\mathcal{X})\sqrt{-8\ln{(\alpha/4)}} \leq \frac{1}{\sqrt{8}}C(\mathcal{X})\sqrt{-8\ln{(\alpha/4)}}$$
	for all $k \geq 2$.
\end{prop}
\begin{proof}
	For any given $\mu(k)$, consider the null distribution $\mathcal{D}_0$ given by linear program \eqref{eq:limdist_null_expression} and bound its objective (everywhere) as 
	\begin{equation*}
		\begin{aligned}
			\sum_{i=1}^k \sqrt{a_i}\langle u_i,G_i  \rangle  \leq  \sum_{i=1}^k \sqrt{a_i} \left| \langle u_i, G_i  \rangle\right| &  \leq \sum_{i=1}^k \sqrt{a_i} \| u_i\| \|G_i\| \\
			&\underset{\text{Lemma } \ref{lem:dual_bd_null}}{\leq} \left(\sum_{i=1}^k \sqrt{a_i}\right)\frac{k-1}{k^2}C(\mathcal{X})\|G_1 \|
		\end{aligned}
	\end{equation*}
	Thus the optimal value $X_0$ of \eqref{eq:limdist_null_expression} is also bounded by the same quantity. 
	
	Desired cut-off $c_\alpha(\mathcal{X})$ is obtained using the cut-off $t_\alpha$ for the distribution of $\|G_1\|$. To define $t_\alpha$, we use the following concentration result from \cite{Ledoux}:
	\vspace{0.3 cm}
	\par{\textbf{Concentration of $\|G_1 \|$}(\textbf{Equation (3.5) from \cite{Ledoux}})} For a centered Gaussian random variable $G_1$, given any $t>0$,
	\begin{equation}\label{eq:large_dev}
		\mathbb{P}\left( \|G_1 \| \geq t  \right) \leq 4 e^{-\frac{t^2}{8 \mathbb{E}\| G_1\|^2}}
	\end{equation}
	\vspace{0.2 cm}
	
	Recalling that $G_1 = ({G_1}_1,\hdots,{G_1}_N)$ with $\text{Cov}(G_1) = \text{diag}(\mu_1) - \mu_1\mu_1'$ where $\mu_1 = (p_1,\hdots,p_N)$, we get that
	\begin{equation*}
		\begin{aligned}
			\mathbb{E}\| G_1\|^2 &= \mathbb{E}\left[ ({G_1}_1)^2 + \cdots + ({G_1}_N)^2  \right] \\
			& = p_1(1-p_1) + \cdots + p_N(1-p_N) \\
			& = 1 - \sum_{i=1}^N p_i^2 < 1
		\end{aligned}
	\end{equation*}
	This gives that $4 e^{-\frac{t^2}{8 \mathbb{E}\| G_1\|^2}} \leq  4 e^{-\frac{t^2}{8}}$, which allows to bound the probability in (\ref{eq:large_dev}) as 
	$$\mathbb{P}\left( \|G_1 \| \geq t  \right) \leq 4 e^{-\frac{t^2}{8}}$$
	With $t=\sqrt{-8\ln{(\alpha/4)}}$, this results in  $\mathbb{P}\left( \|G_1 \| \geq t  \right) \leq \alpha$. Note that the result holds for any $\mu_1$. 
	
	So we let $t_\alpha:=\sqrt{-8\ln{(\alpha/4)}}$, which ensures $\mathbb{P}\left( \|G_1 \| \geq t_\alpha  \right) \leq \alpha$. With this choice, we indeed get that
	\begin{equation*}
		\begin{aligned}
			&\mathbb{P}\left(  X_0 \geq \left[\sum_{i=1}^k \sqrt{a_i}\right]\frac{k-1}{k^2}C(\mathcal{X})  \cdot t_\alpha \right) \\
			\leq &\mathbb{P}\left( \left[\sum_{i=1}^k \sqrt{a_i}\right]\frac{k-1}{k^2}C(\mathcal{X}) \| G_1\| \geq \left[\sum_{i=1}^k \sqrt{a_i}\right]\frac{k-1}{k^2}C(\mathcal{X})  \cdot t_\alpha \right) \\
			=&  \mathbb{P}\left( \|G_1 \| \geq t_\alpha  \right) \leq \alpha 
		\end{aligned}
	\end{equation*}
	Define $c_\alpha(\mathcal{X}):=\left[\sum_{i=1}^k \sqrt{a_i}\right]\frac{k-1}{k^2}C(\mathcal{X})  \cdot t_\alpha = \left[\sum_{i=1}^k \sqrt{a_i}\right]\frac{k-1}{k^2}C(\mathcal{X})\sqrt{-8\ln{(\alpha/4)}}$. 
	By the above, we have $\mathbb{P}\left(X_0 \geq c_\alpha(\mathcal{X})\right) \leq \alpha$, and hence the original test cut-off satisfies $c_{\alpha,\mathcal{D}_0} \leq c_\alpha(\mathcal{X})$, which holds for any $\mathcal{D}_0$.
	
\end{proof}

\subsection{Consistency and power}\label{subsec:power}

\begin{figure}
	\includegraphics[width=0.8\textwidth]{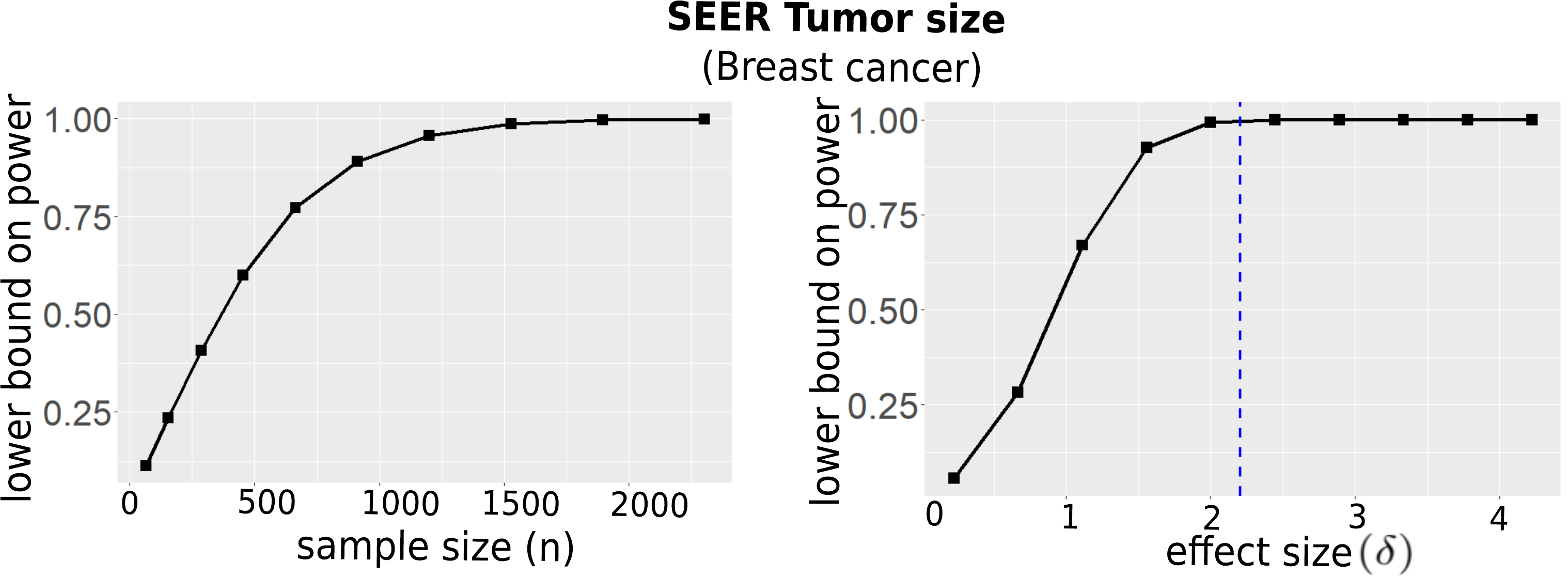}
	\caption{Illustration of theoretical power guarantees for $k=2$ case  (Theorem \ref{thm:unif_consistency_F} and Corollary \ref{coro:power}) for real data settings described in Section \ref{subsec:real_1D}. The test of $H_0$ compares distributions of tumor size for two groups of patients (marked ``alive, no metastases" and ``dead" in Figure \ref{fig:sizes_1D}B). Left: The effect size (Wasserstein squared distance $\delta$) is fixed at the value calculated from the data (blue line in the right panel), and the lower bound on power is shown as a function of $n$. Right: The sample size ($n$) was fixed at high value (actual sample size in the dataset is even larger), and power is shown as a function of $\delta$. }
	\label{fig:power_real}
\end{figure}

We start by showing the basic requirement for tests comparing $k\geq 2$ measures - consistency under any fixed alternative - by showing that the power tends to $1$ with increasing sample sizes. 

The power of the test  \eqref{eq:test} is
\begin{equation}
	\begin{aligned}[b]
		\text{Power}(\mu) & = \mathbb{P}_{\mathcal{D}_a}\left(\phi_{n,\mu}=1   \right) \\ 
		& = \mathbb{P}_{\mathcal{D}_a}\left(\rho_n(MOT(\widehat{\mu}_n) - 0) \geq c_{\alpha,\mathcal{D}_0} \right) \\
		& = \mathbb{P}_{\mathcal{D}_a} \left( \rho_n(MOT(\widehat{\mu}_n) - MOT(\mu)) \geq c_{\alpha,\mathcal{D}_0} - \rho_n MOT(\mu) \right)  \\
		& = \mathbb{P}\left( X_a \geq c_{\alpha,\mathcal{D}_0} - \rho_n MOT(\mu)  \right) 
	\end{aligned}
	\label{eq:power}
\end{equation}
\begin{remark} To obtain expression of the power for the test in \cite{Sommerfeld}, $MOT(\mu)$ in the above expression can be replaced by $W_2^2(\mu)$\footnote{This choice corresponds to $p=2$-Wasserstein distance in \cite{Sommerfeld}. Other choices of $p \in [1,\infty)$ can be used by adjusting the details accordingly.}, with necessary adjustments for the $k=2$ measures case. 
\end{remark}

Proposition \ref{prop:consistency_fixed} below proves consistency under fixed alternatives for any $k\geq 2$. The proof utilizes a Normal Lower Bound guaranteed by Theorem \ref{thm:as_dist}(c) to lower bound the power of the test.
\begin{prop}[\textbf{Consistency under fixed alternatives, $k\geq 2$ measures}]\label{prop:consistency_fixed}
	Under any given alternative $\mu=(\mu_1,\hdots,\mu_k)$, $k\geq 2$, the test in \eqref{eq:test} satisfies
	$$\lim_{n \to \infty}\mathbb{P}(\phi_{n,\mu} = 1) = 1$$
\end{prop}
\begin{proof}
	Given an alternative $\mu$,  consider a set of dual optimal solutions $\Phi^*$ \eqref{eq:dualsol} to $MOT(\mu)$. Choose any $u^* \in \Phi^*$ (treated fixed after the choice), and consider the Normal Lower Bound $NLB_{u^*} \sim \mathcal{N}\left(0,\sum_{i=1}^k a_i{u_i^*}'\Sigma_i u_i^*\right)$ guaranteed by Theorem \ref{thm:as_dist}(c). Using \eqref{eq:power}, we have that
	\begin{equation*}
		\begin{aligned}
			\mathbb{P}_{\mathcal{D}_a}\left(\phi_{n,\mu}=1   \right) & = \mathbb{P}\left( X_a \geq c_{\alpha,\mathcal{D}_0} - \rho_n MOT(\mu)  \right)\\
			& \geq \mathbb{P}\left( NLB_{u^*} \geq c_{\alpha,\mathcal{D}_0} - \rho_n MOT(\mu)\right)
		\end{aligned}
	\end{equation*}
	Since $\rho_n \to \infty$ while the test's cut-off $c_{\alpha,\mathcal{D}_0}$ and the true value $MOT(\mu)>0$ do not change with $n$, we get that $c_{\alpha,\mathcal{D}_0} - \rho_n MOT(\mu) \to -\infty$, and hence the Gaussian random variable $NLB_{u^*}$ exceeds it with probability tending to $1$ as $n \to \infty$. 
\end{proof}
Next, we prove uniform consistency of the test \eqref{eq:test} over a broad class of alternatives. We start with a case of $k=2$ measures in Theorem \ref{thm:unif_consistency_F}, which proves uniform consistency of the test proposed by \cite{Sommerfeld}. We then move to general $k\geq 2$ ($k$ is allowed to change) in Proposition \ref{prop:unif_consist_FK}, concluding uniform consistency of tests of this type.

Our results are proved under the following assumption (B1) below, which is discussed in Remark \ref{rmk:slackness}. This assumption is expected to hold when measures are not too far from each other, and the power without this assumption is expected to be higher. Removing (B1) poses the difficulty of bounding dual solutions $u$ 
uniformly over alternative polytopes $\{\mu'u=MOT(\mu), A'u \leq c \}$; we use such bound to control alternative variances. Under (B1), the condition $\sum_{i=1}^k u_i = 0$ allows to bound $u$ (and hence alternative variances) explicitly and uniformly over $\mu$.  
\begin{assumption}[B1]\label{assn:dual_opposite} There exist dual solutions $u$ to $MOT(\mu)$ satisfying $\sum_{i=1}^k u_i = 0$.
\end{assumption}

For any fixed metric space $\mathcal{X}$ with $N$ points and any $\delta > 0$, define the class of alternatives
$$ \mathcal{F}(\delta):=\{\mu \text{ on } \mathcal{X}: W^2_2(\mu) \geq \delta \}$$

\begin{theorem}[\textbf{Uniform consistency, $k=2$ measures}]\label{thm:unif_consistency_F}
	For $k=2$, the test \eqref{eq:test} (or, equivalently, the test based on Theorem 1(c) of \cite{Sommerfeld}) satisfies
	$$\lim_{n \to \infty} \inf_{\mu \in \mathcal{F}(\delta) \cap (B1)} \mathbb{P}(\phi_{n,\mu} = 1) = 1$$
\end{theorem}
\begin{proof}
	Fix test level $\alpha \in (0,1)$. The goal is to show that, independently of the nature of the alternative $\mu \in \mathcal{F}(\delta)\cap(B1)$, the probability that the test rejects $H_0$ given by 
	$$\mathbb{P}(\phi_{n,\mu} = 1) = \mathbb{P}\left( X_a \geq c_{\alpha,\mathcal{D}_0} - \rho_n \frac{1}{4} W_2^2(\mu)  \right)$$
	tends to $1$ as $n\to \infty$ (the factor of $\frac{1}{4}$ is due to $MOT(\mu) = \frac{1}{4}W_2^2(\mu)$ for $k=2$).  
	
	Assume for simplicity of notation that the sample sizes are equal, i.e.  $n_1=\cdots=n_k$\footnote{The proof for unequal samples sizes would be similar by considering $n:=\min_i\{n_i\}$.}. Note first that the null cut-offs $c_{\alpha,\mathcal{D}_0}$ can be bounded above uniformly over the class $\mathcal{F}(\delta)\cap(B1)$ using Proposition \ref{prop:c_alpha_bd}, which with $k=2$ gives the bound
	$$c_{\alpha,\mathcal{D}_0} \leq \frac{1}{\sqrt{8}}C(\mathcal{X})\sqrt{-8\ln(\alpha/4)}:= c_\alpha(\mathcal{X})$$
	Hence, 
	$$c_{\alpha,\mathcal{D}_0} - \rho_n \frac{1}{4}W_2^2(\mu) \leq c_\alpha(\mathcal{X}) - \rho_n \frac{\delta}{4} \text{ for all } \mu \in \mathcal{F}(\delta)\cap(B1)$$
	This expression represents the ``worst" (over $\mathcal{F}(\delta)\cap(B1)$) value that any given $X_a$ must exceed to give the test a power.
	
	Next we show that any $X_a$ will exceed this bound with probability tending to $1$ as $n \to \infty$. To this end, let $\mu \in \mathcal{F}(\delta)\cap(B1)$, and consider the corresponding alternative distribution of  $X_a$. Consider the dual solutions $({u_1}_a^*, {u_2}_a^*)$ satisfying Assumption (B1) and the corresponding Normal Lower Bound 
	$$NLB_{u_a^*} \sim \mathcal{N}\left(0,\frac{1}{2}{{u_1}_a^*}'\Sigma_1 {{u_1}_a^*} + \frac{1}{2}{{u_2}_a^*}'\Sigma_2{{u_2}_a^*}\right)$$
	guaranteed by Theorem \ref{thm:as_dist}(c). Note that its variance 
	\begin{equation*}
		\begin{aligned}
			\frac{1}{2}{{u_1}_a^*}'\Sigma_1 {u_1}_a^* + \frac{1}{2}{{u_2}_a^*}' \Sigma_2 {u_2}_a^* \leq \frac{1}{2} & \| {u_1}_a^*\|^2 \cdot \lambda_{max}(\Sigma_1) + \frac{1}{2}\|{u_2}_a^* \|^2 \cdot \lambda_{max}(\Sigma_2) \\
			\underset{(B1)}{=} & \frac{1}{2}\| {u_1}_a^*\|^2 \left( \lambda_{max}(\Sigma_1) + \lambda_{max}(\Sigma_2)  \right)
		\end{aligned}
	\end{equation*}
	where $\lambda_{max}(\cdot)$ denotes the largest eigenvalue of a matrix argument. 
	
	Using Theorem 1 of \cite{Benasseni}, the eigenvalues of $\Sigma_1,\Sigma_2$ are upper bounded by entries of $\mu_1$ and $\mu_2$ as $\lambda_{max}(\Sigma_1)\leq \max_i({\mu_1})_i$ and $\lambda_{max}(\Sigma_2)\leq \max_i ({\mu_2})_i$, with the uniform upper bound of $1$ for all instances $\mu \in \mathcal{F}(\delta)$\footnote{In fact, the bound hold for any $\mu$ and is not restricted to the class $\mathcal{F}$ - see \cite{Benasseni}.}. Hence,
	$$\lambda_{max}(\Sigma_1) + \lambda_{max}(\Sigma_2) \leq 2$$
	providing a uniform upper bound on the eigenvalue part. Further, $\| {u_1}_a^* \| \leq \frac{1}{4}C(\mathcal{X})$ by Lemma \ref{lem:dual_bd_null}. Thus, letting 
	\begin{equation}\label{eq:sigma_bd}
		\sigma^2(\mathcal{X}):=\left(\frac{1}{4}C(\mathcal{X})\right)^2
	\end{equation}
	uniformly bounds the variances for all $NLB$'s chosen as above for $X_a$'s arising from $\mu \in \mathcal{F}(\delta)\cap (B1)$.  
	
	The final step is to combine the above uniform bounding arguments to get the power $\to 1$. Note that for large enough $n$, $c_\alpha(\mathcal{X}) - \rho_n \frac{1}{4}\delta < 0$, and also that $c_\alpha(\mathcal{X}) - \rho_n \frac{1}{4}\delta \to -\infty$ as $n \to \infty$. Hence, we have that, for any $\mu \in \mathcal{F}(\delta)\cap(B1)$, for large enough $n$ depending only on $\delta$ and $\mathcal{X}$ but not the nature of $\mu$,
	\begin{equation}
		\begin{aligned}[b]
			\mathbb{P}(\phi_{n,\mu} = 1) &= \mathbb{P}\left( X_a \geq c_{\alpha,\mathcal{D}_0} - \rho_n \frac{1}{4}W_2^2(\mu)  \right)\\ 
			& \geq \mathbb{P}\left( X_a \geq c_\alpha(\mathcal{X}) - \rho_n \frac{\delta}{4} \right) \\
			& \geq \mathbb{P}\left(NLB_{u^*_a} \geq c_\alpha(\mathcal{X}) - \rho_n \frac{\delta}{4}  \right) \\
			& \geq \mathbb{P}\left( \mathcal{N}\left(0, \sigma^2(\mathcal{X})\right) \geq c_\alpha(\mathcal{X}) - \rho_n \frac{\delta}{4}  \right)
		\end{aligned}
		\label{eq:NLB_bd_F1}
	\end{equation}
	which tends to $1$ with $c_\alpha(\mathcal{X}) - \rho_n \frac{\delta}{4} \to -\infty$ as $n\to \infty$. This gives the uniform lower bound on the power over $\mathcal{F}(\delta)\cap(B1)$ proving the uniform consistency of the test over this broad class of alternatives. 
\end{proof}

Using Theorem \ref{thm:unif_consistency_F}, one can provide a practical lower bound on the power as a function of the sample size $n$ and/or the effect size $\delta$ when measures are supported on a known metric space $\mathcal{X}$. Note that the dual bound $C(\mathcal{X})$ is rather conservative; it can be replaced by a bound on $\| u_1\|$ computed for the polytope $\{u_1+u_2 = 0, A'u \leq c \}$ in every specific case. To find these bounds, one could solve linear programs $\{ \max/\min u^i: u_1+u_2 = 0, A'u \leq c \}$ for each entry $u^i$ of $u$ to estimate magnitudes of the dual variables over a given polytope. Denoting the resulting bound $\widetilde{C}(\mathcal{X})$, we have

\begin{coro}[\textbf{Lower bound on the power of the two-sample test}]\label{coro:power} For any alternative $\mu \in \mathcal{F}(\delta) \cap (B1)$, the two-sample test \eqref{eq:test} (or the test based on Theorem 1(c) of \cite{Sommerfeld}) with equal sample sizes $n$ has
	$$\text{power} \geq 1-\Phi\left(\frac{4 \widetilde{C}(\mathcal{X}) \sqrt{-\log(\alpha/4)}) - \frac{\sqrt{n}}{\sqrt{2}}\delta}{\widetilde{C}(\mathcal{X})}   \right)$$
	where $\Phi(\cdot)$ denotes the cumulative distribution function of the standard Normal distribution.
\end{coro}
Illustration of this bound on the real data from Section \ref{subsec:real_1D} is provided in Figure \ref{fig:power_real}.

Using techniques similar to the proof of Theorem \ref{thm:unif_consistency_F}, it is possible to prove uniform consistency of the test \eqref{eq:test} for alternatives $\mu(k):=(\mu_1,\hdots, \mu_k)$ in the class 
\begin{equation*}
	\mathcal{F}_K(\delta):=\{\mu(k) \text{ on } \mathcal{X}:  k\leq K, \text{ }MOT(\mu)\geq \delta \}
\end{equation*}
defined for any $\delta>0$ and any fixed $2 \leq K < \infty$ under the Assumption (B1). This gives consistency of the test \eqref{eq:test} uniformly over alternatives with $k\leq K$ measures:
\begin{prop}[\textbf{Uniform consistency in the class $\mathcal{F}_K(\delta)$}]\label{prop:unif_consist_FK} The test in \eqref{eq:test} satisfies
	$$\lim_{n \to \infty} \inf_{\mu(k) \in \mathcal{F}_K(\delta) \cap (B1)} \mathbb{P}(\phi_{n,\mu} = 1) = 1$$
\end{prop}
\begin{proof}
	
	Similarly to the proof of Theorem \ref{thm:unif_consistency_F}, the null cut-offs are uniformly bounded using Proposition \ref{prop:c_alpha_bd} as $c_{\alpha,\mathcal{D}_0} \leq c_\alpha(\mathcal{X})$. Recall that (taking equal sample sizes $n$ for simplicity) for any $2 \leq k \leq K$, $\rho_n = \sqrt{\frac{n}{k^{k-1}}}$, and hence $c_\alpha - \rho_n MOT(\mu(k)) < 0$ for $n$ large enough to ensure this holds for all $k\leq K$.
	
	Similarly to the proof of Theorem \ref{thm:unif_consistency_F}, each alternative random variable $X_a$ arising from $\mu(k)$ has a Normal Lower Bound 
	$$NLB_{u_a^*} \sim \mathcal{N}\left(0,\frac{1}{k^{k-1}}{{u_1}_a^*}'\Sigma_1 {{u_1}_a^*} + \cdots + \frac{1}{k^{k-1}}{{u_k}_a^*}'\Sigma_k{{u_k}_a^*}\right)$$
	with the variance bounded above by 
	\begin{equation}\label{eq:sigsq}
		\sigma^2(\mathcal{X},k):=\frac{1}{k^{k-1}}\cdot\left( \frac{k-1}{k^2} C(\mathcal{X})\right)^2 \cdot k = \frac{(k-1)^2}{k^{k+2}}\left(C(\mathcal{X})\right)^2
	\end{equation}
	which decreases with $k$. Hence, it is bounded uniformly in $k$ by $\sigma^2(\mathcal{X})$ from the $k=2$ case (equation \eqref{eq:sigma_bd}), and the rest of the argument carries in exactly the same way as in the proof of Theorem \ref{thm:unif_consistency_F}.
\end{proof}

\begin{remark}[Large $k$ and connection with \cite{Kim}]\label{rmk:uniform_consistency_change_K}
	In practice, the upper bound on the number $k$ of measures in $\mathcal{F}_K(\delta)$ cannot be too large. This is due to the requirement $\rho_n=\sqrt{\frac{n}{k^{k-1}}} \to \infty$, which forces very large sample sizes that may not be practically plausible. While this limitation is natural for asymptotic $k$-sample tests that work with fixed $k$ (as discussed, e.g., in  \cite{Zhan}), recent results of \cite{Kim} show that permutation approach for certain test statistics allows for growing $k$ and $n$ simultaneously. More precisely, in the class of alternatives where only a few measures differ from the rest of the collection, the permutation kernel based test is uniformly powerful if the population version of the test statistic $\delta$ sufficiently exceeds $\sqrt{\frac{\log k}{n}}$. Below we discuss sequences of alternatives whose $MOT$ population value does not decrease with $k$, and hence the $MOT$ test statistic is expected to perform well in a permutation procedure. We leave a theoretical power analysis concerning $MOT$ permutation test for future work. 
\end{remark}

The ``clustered" alternatives are collections $\mu=(\mu_1,\hdots,\mu_k)$ that separate into two groups (or ``clusters"), with $k/C$ measures in each cluster that are all the same. Such situation might arise, for example, if an applied treatment causes $C$ different types of responses. For instance, for $C=2$,  define
\begin{equation*}
	\begin{aligned}\mathcal{F}^2_k:=\{\mu \text{ on } \mathcal{X}: \mu_1=\cdots=\mu_{k/2}\neq \mu_{k/2+1} =\cdots=\mu_k\}
	\end{aligned}
\end{equation*}
(see Figure \ref{fig:synth_3D}B for illustration). The classes $\mathcal{F}^C_k$ with $C=3,\cdots,k$ are defined analogously (in each case, $k$ is assumed to be divisible by $C$). 

\begin{lemma}[\textbf{MOT values for ``clustered" alternatives}]\label{lem:clustered_alt} For $\mu \in \mathcal{F}^2_k$, we have $MOT(\mu) = MOT(\mu_1,\mu_k) = \frac{1}{4} W_2^2(\mu_1,\mu_k)$. More generally, for $\mu \in \mathcal{F}^C_k$, $MOT(\mu) = MOT(\{\mu_i\}_{i=1}^C)$, where $\{\mu_i\}_{i=1}^C$ is a collection consisting of one measure from each cluster. 
\end{lemma}
Proof is provided in Appendix \ref{app:proof_clustered_alt}. Note that true $MOT$ values for ``clustered" alternatives do not decrease with increasing number of measures and thus may serve as a suitable test statistics in permutation tests against alternatives in $\mathcal{F}^C_k$.

Finally, we comment on the $MOT$ values in a ``sparse" alternative class when only one measure is different from the rest (alternatives of this type are considered in both \cite{Kim} and \cite{Zhan}:
$$\mathcal{F}^s_{k}:=\{\mu \text{ on } \mathcal{X}: \mu_1=\cdots=\mu_{k-1}\neq \mu_k\}$$
While $MOT$ values do decrease with $k$ in this sequences of alternatives, we can state precisely how the rate of this decrease is controlled  (proved in Appendix \ref{app:proof_sparse_alt}):

\begin{lemma}[\textbf{MOT values for ``sparse" alternatives}]\label{lem:sparse_alt} For $\mu \in \mathcal{F}^s_k$, $MOT(\mu) = \frac{k-1}{k^2}$ $W_2^2(\mu_1,\mu_k)$.
\end{lemma}

Empirical performance of asymptotic $MOT$ test \eqref{eq:test} and permutation $MOT$ test (Section \ref{sec:perm}) on ``clustered" and ``sparse" alternatives are illustrated in Figure \ref{fig:synth_3D}.

\section{Sampling from null and alternative distributions}\label{sec:sampling}
\begin{figure}
	\includegraphics[width=0.8\textwidth]{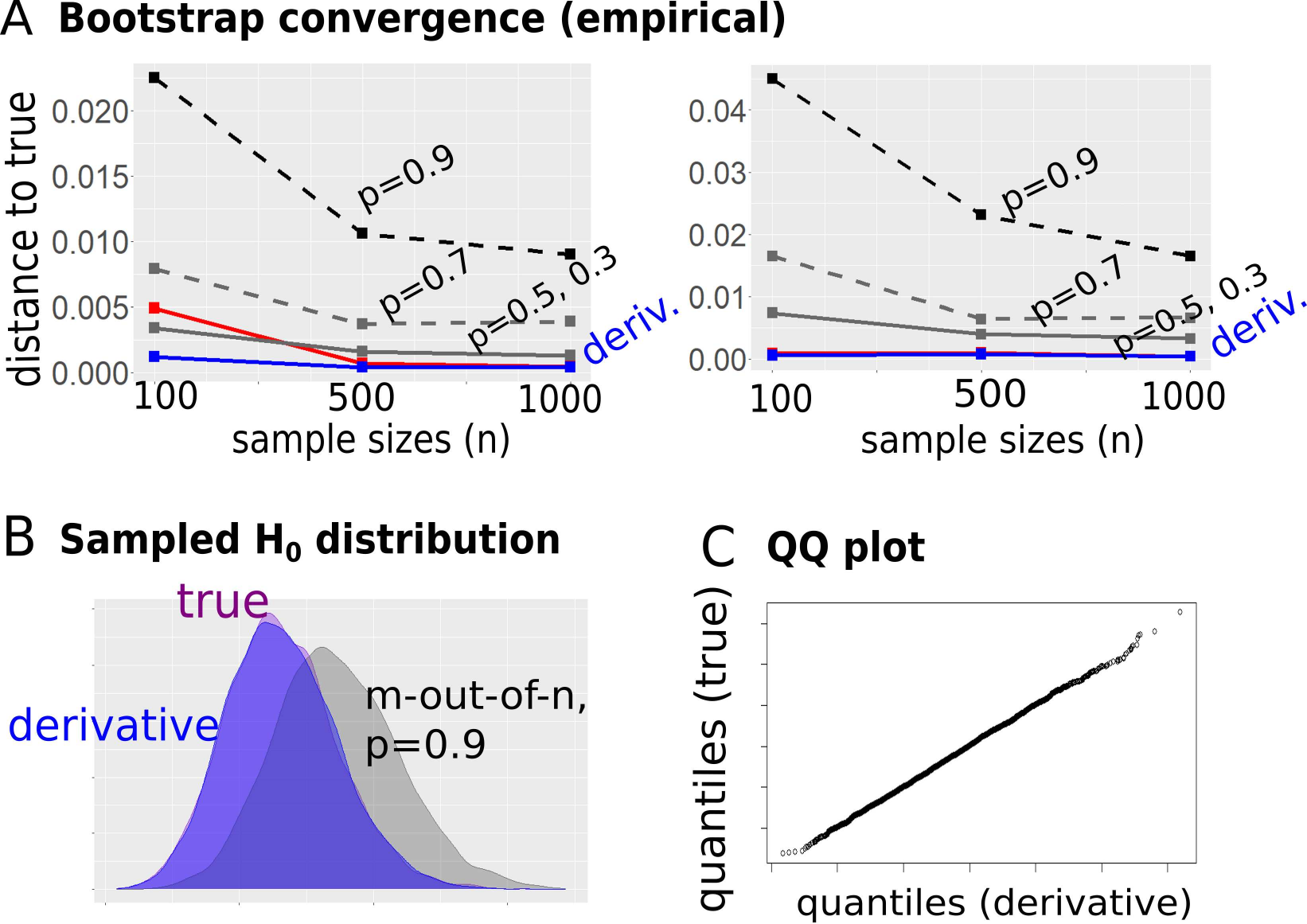}
	\caption{Illustration of bootstrap consistency (Section \ref{sec:boot}). A. Convergence of m-out-of-n and derivative bootstrap sampling distributions (both sampled based on the empirical $\widehat{\mu}_n$) to the true null distribution (sampled based on the true $\mu$) assessed by 1-Wasserstein distance. The m-out-of-n bootstrap schemes are shown with $m:=n^p$, $p \in \{0.3,0.5,0.7,0.9 \}$. Observed convergence rate  is fastest for the derivative bootstrap (blue), and is slower for the m-out-of-n bootstrap for  larger values of $m$. The data is based on the \textbf{3D Experiment} dataset (Figure \ref{fig:synth_3D}) by choosing bottom unit square in $2D$ (Left panel here) and a unit square in $3D$ (Right panel here).   B. Sampling distributions corresponding to panel A, $n=500$. C. Quantile-quantile plot illustrating closeness of the derivative bootstrap to the true distribution from panel B.}
	\label{fig:boot_conv}
\end{figure}

\begin{figure}
	\includegraphics[width=1\textwidth]{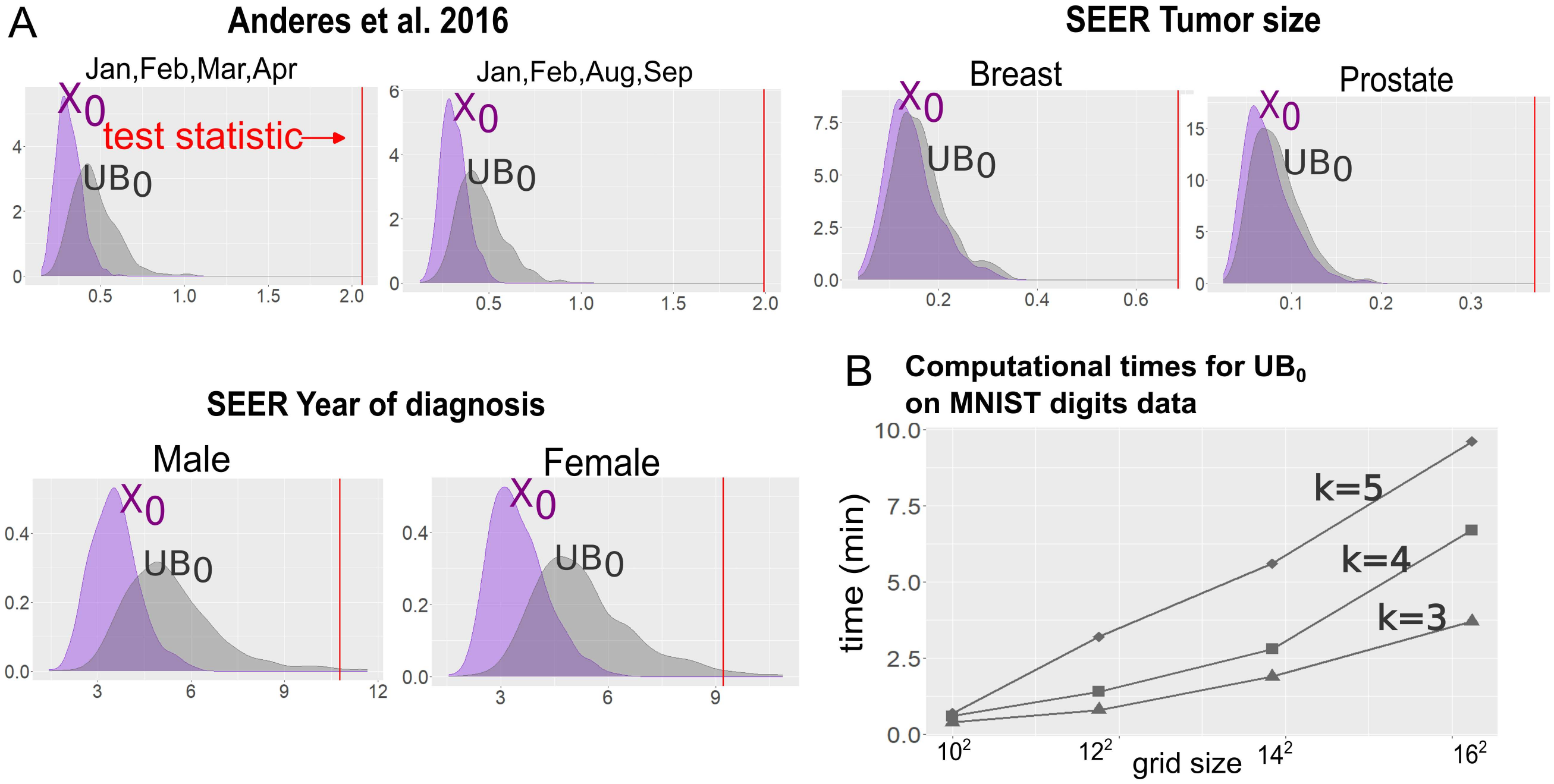}
	\caption{Illustration of performance of the null upper bound $UB_0$ (Section \ref{sec:null_approximation}). A.  $H_0$ testing using upper bound $UB_0$ and the true null distribution $X_0$ for all three real (or real-based) datasets considered in this paper (Section \ref{sec:applications}).  Observe that $UB_0$ produces the same conclusion as $X_0$ (rejection of $H_0$) on all considered datasets, while having much lower computational complexity (Table \ref{table:complexity}). The whole analysis took under 5 minutes on the standard laptop, with negligible fraction of time taken by $UB_0$.  B. Times to compute $500$ samples from $UB_0$ for (subsampled grid) MNIST digits data \cite{MNIST} for $k$ images on a standard laptop. All simulations were conducted on AMD Ryzen 5 7520U with 16 GB RAM.  }
	\label{fig:ub}
\end{figure}

\subsection{Bootstrap: $m$-out-of-$n$ and derivative}\label{sec:boot}
We recall that the limiting laws in Theorem \ref{thm:as_dist} depend on the true measures $\mu$, similarly to the $k=1,2$-sample cases considered in \cite{Sommerfeld} and \cite{Hundrieser}. More precisely, the laws are of the form
$$\rho_n \left(f(\widehat{\mu}_n) - \ f(\mu)  \right) \xrightarrow[]{\text{in law}} f'_\mu(G)$$
where $f: \mu \longrightarrow MOT(\mu)$ is the map with Hadamard directional derivative $f'$ at $\mu$ in the direction of $G \overset{\text{in law}}{=} \lim_n \rho_n \left(\widehat{\mu}_n - \mu \right)$ and $n:=(n_1,\hdots,n_k)$. The classical bootstrap estimator of $f_\mu(G)$ in a sense of \cite{Efron} would be constructed by sampling from the conditional (given the data) law of $\rho_n \left(f(\hat{\mu}_n^*) - f(\widehat{\mu}_n)   \right)$, where $\hat{\mu}_n^*$ is obtained by taking $n$ samples from the vector of empirical measures $\widehat{\mu}_n$. By Theorem 3.1 of \cite{Fang}, this estimator is not consistent when $f'_\mu(G)$ is non-Gaussian, which is always the case under $H_0$ and frequently under $H_a$. 

In place of inconsistent classical bootstrap, \cite{Fang} proposes a consistent  bootstrap procedure to estimate the law of $f'_\mu(G)$. The approach of \cite{Fang} is to ensure consistency of an estimator $f'_n(\cdot)$ of the map $f'_\mu(\cdot)$ uniformly in the argument $(\cdot)$, assuming that the law of the argument $(\cdot)$ is estimated by (some) consistent bootstrap scheme. Two different choices for $f'_n(\cdot)$ then lead to bootstrap schemes frequently termed \textit{m-out-of-n} and \textit{derivative} bootstrap methods, respectively (see Section 1 of \cite{Fang} on historical notes on these methods).  

The work of \cite{Sommerfeld} outlines the consistency results for these two schemes in $k=1,2$-sample cases. For completion, we describe these schemes in the general case of $k\geq 2$ (proved in Appendix \ref{app:proof_boot}):

\begin{prop}[\textbf{Consistency of bootstrap from \cite{Fang}}] \label{prop:boot} 
	The results of parts (a) and (b) concern with two estimators $f'_n(\cdot)$ of the map $f_\mu'(\cdot)$. 
	\begin{itemize}
		\item[(a)] $f'_n: h \longrightarrow f_n'(h)$ given by
		$$f'_n(h) := \frac{f(\widehat{\mu}_n + \varepsilon_n h) - f(\widehat{\mu}_n)}{\varepsilon_n}$$
		composed with the estimator of $G$ given by 
		$$\hat{G}^* := \rho_m\left(\widehat{\mu}^*_{(m)} - \widehat{\mu}_n\right)$$
		results in a consistent bootstrap estimator $f'_n(\hat{G}^*)$ of $f_\mu'(G)$ under both $H_0$ and $H_a$. Here, $\hat{\mu}^*_{(m)}$ is obtained by resampling m out of n observations from $\widehat{\mu}_n$ with $m:=\sqrt{n}$, and $\varepsilon_n \to 0$ such that $\rho_n \varepsilon_n \to \infty$. 
		
		\underline{Note:} The choice $\varepsilon_n := \frac{1}{\rho_m}$ leads to
		$$f'_n (\hat{G}^*) = \rho_m \left(f(\widehat{\mu}^*_{(m)}) - f(\widehat{\mu}_n)  \right)$$
		which is frequently termed the m-out-of-n bootstrap estimator of $f_\mu'(G)$ and is considered in \cite{Sommerfeld} and \cite{Hundrieser} for the Wasserstein distance map $f$. 
		\item[(b)] $f_n': h \longrightarrow f_{\hat{\mu}_n}'(h)$ given by
		\begin{equation*}
			\begin{aligned}
				f'_{\hat{\mu}_n}(h):= & \max_{u} \sum_{i=1}^k \langle u_i, h_i \rangle \\
				& \langle u,\hat{\mu}_n  \rangle = MOT(\hat{\mu}_n) \\
				& A'u \leq c
			\end{aligned}
		\end{equation*}
		composed with the estimator of $G$ given by
		$$\hat{G}^* := \left(\sqrt{\hat{a}_1} \hat{G}_1^*, \hdots,    \sqrt{\hat{a}_k} \hat{G}^*_k\right)$$
		where each $\hat{G}_i^*$ is $N\left(0,\text{diag}(\hat{\mu}^*_1) - \hat{\mu}_1^*(\hat{\mu}_1^*)' \right)$ results in a consistent bootstrap estimator $f_n'(\hat{G}^*)$ of $f_\mu'(G)$ under $H_0$.
		
		\underline{Note:} This estimator is frequently termed the derivative bootstrap estimator of $f'_\mu(G)$ and is considered in \cite{Sommerfeld} for the Wasserstein distance map $f$. 
	\end{itemize}
\end{prop}

Pseudocodes \ref{pseudo:mn} and \ref{pseudo:der} describe sampling from the limiting laws of $MOT$ under $H_0$ and $H_a$ using bootstrap schemes in Proposition \ref{prop:boot}.

\begin{pseudocode}[\textbf{m-out-of-n bootstrap to obtain one sample from $H_0$ or $H_a$ limiting law}] \label{pseudo:mn}
	Given the data $\widehat{\mu}_n$,
	\begin{enumerate}
		\item Let $m_i:=n_i^p$, $p \in (0,1)$.
		\item For each $i=1,\hdots,k$, \\  sample $\widehat{\mu}_i^* \sim \text{Multinomial}(m_i,\widehat{\mu}_1)$ under $H_0$, or \\  sample $\widehat{\mu}_i^* \sim \text{Multinomial}(m_i,\widehat{\mu}_i)$ under $H_a$.
		\item Compute $MOT(\widehat{\mu}_1^*,\hdots,\widehat{\mu}_k^*)$ by solving the program \eqref{eq:mot_primal}.
		\item Report $\rho_m  MOT(\widehat{\mu}_1^*,\hdots,\widehat{\mu}_k^*)$ under $H_0$ or \\ $\rho_m \left(MOT(\widehat{\mu}_1^*,\hdots,\widehat{\mu}_k^*) - MOT(\widehat{\mu}_n) \right)$ under $H_a$, \\
		where $\rho_m=\frac{\sqrt{m_1 \cdot \hdots \cdot m_k}}{\left(\sqrt{m_1 + \hdots + m_k}  \right)^{k-1}}$.
	\end{enumerate}
\end{pseudocode}

\begin{pseudocode}[\textbf{Derivative bootstrap to obtain one sample from $H_0$ limiting law}] \label{pseudo:der}
	Given the data $\widehat{\mu}_n$,
	\begin{enumerate}
		\item Sample $\hat{\mu}_1^* \sim \text{Multinomial}(n_1,\widehat{\mu}_1)$.
		\item For each $i=1,\hdots,k$, \\
		sample $\hat{G}_i^* \sim \text{Normal}(0,\text{diag}(\widehat{\mu}_1^*) - \hat{\mu}_1^* (\hat{\mu}_1^*)')$. \\
		Let $a_i = \prod_{j \neq i} \hat{\lambda}_j$, where $\hat{\lambda}_j = \frac{n_j}{n_1+\hdots+n_k}$.
		\item Solve the program \eqref{eq:limdist_null_expression} with $\{\hat{G}_i^*\}_{i=1}^k$ in place of $\{G_i\}_{i=1}^k$. 
	\end{enumerate}
\end{pseudocode}

\subsubsection{Computational complexity of bootstrap} \label{sec:boot_complexity}
For the m-out-of-n bootstrap, computation in Step 3 requires solving the primal $MOT$ program \eqref{eq:mot_primal}, which is a linear program with $N^k$ variables, i.e. exponentially many in terms of the cardinality of the support space $\mathcal{X}$. By strong duality, the optimal value of primal MOT program is the same as that of the dual MOT \eqref{eq:mot_dual}, which is a linear program 
$$\max_u \langle u,\mu \rangle \text{ s.t. } A'u\leq c$$
with polynomially many variables but exponentially many constraints. It is well-known that a linear program with exponentially many constraints can be proved polynomial time solvable via ellipsoid method provided its feasible set $\{A'u\leq c\}$ has a polynomial time computable separation oracle (see, e.g. Section 8.5 of \cite{Bertsimas}). Such oracle is a procedure which accepts a proposal point $u \in \mathbb{R}^{kN}$ and either confirms that $u \in \{A'u\leq c \}$, or outputs a violated constraint. Polynomial separation oracle is found for the dual $MOT$ problem in \cite{Altschuler} (Proposition 12), resulting in polynomial time algorithm to solve $MOT$ problem with quadratic cost \eqref{eq:cost}  (Theorem 2 of \cite{Altschuler})\footnote{Besides the optimal value which agrees for the  primal and the dual, \cite{Altschuler} are also interested in the primal vertex solution. For that reason, they also discuss how to get primal solution in polynomial time in their Proposition 11.}.

For the derivative bootstrap, computation in Step 3 requires solving  program \eqref{eq:limdist_null_expression}, which similar to the dual MOT linear program and is given by
$$\max_u \langle u, g \rangle \text{ s.t. } A'u \leq c \text{ and } Bu = 0$$
where $B$ is the matrix of the linear map $u \longrightarrow \sum_{i=1}^k u_i$ and $g$ is a realization of $G^*$ (the nature of the coefficient vector $g$ does not affect the complexity). Note that there are only polynomially many constraints in $B$ (namely, $N$ of them); hence, given a polynomial separation oracle for $\{A'u \leq c\}$, the rest of the constraints in $\{ Bu=0\}$ can be checked in polynomial time, giving the following theoretical complexity for result for the derivative bootstrap for $MOT$:
\begin{lemma}[\textbf{Polynomial complexity of derivative bootstrap}]\label{lem:der_boot_compl} The derivative bootstrap linear program (Step 3 in Pseudocode \ref{pseudo:der}) has computational complexity $\text{poly}(N,k,\log U)$, where $\log U$ is an upper bound on the bits of precision used to represent the coefficient vector $G^*$.
\end{lemma}
Proof details missing from the above discussion are provided in Appendix \ref{app:der_boot_comp}.

Computational complexity of bootstrap methods is summarized in Table \ref{table:complexity}, and consistency of bootstrap is illustrated in Figure \ref{fig:boot_conv}.

\subsection{Fast approximation of the null distribution by  ${UB}_0$} \label{sec:null_approximation}
A fast alternative to the bootstrap sampling from the null distribution is to utilize the lower bound $UB_0$ on the null random variable $X_0$ provided by equation \eqref{eq:UB_0} in Theorem \ref{thm:as_dist}(b). As the proof of the theorem shows, a stochastic upper bound $UB_0$ can be constructed to have $kN(N-1)$ constraints in place of $N^k$ constraints in $X_0$ by exploiting a constraint structure under $H_0$ (see Appendix \ref{app:proof_thm_b} for details). Note that, with only quadratically many constraints, the linear program for $UB_0$ with any realization of the coefficient vector $G$ can be solved fast by modern linear program solvers. 

Sampling from $UB_0$ can be viewed as obtaining an upper bound on the derivative bootstrap sampling distribution of $X_0$, via the following algorithm:

\begin{pseudocode}[\textbf{Sampling $UB_0$}] \label{pseudo:der}
	Given the data $\widehat{\mu}_n$,
	\begin{enumerate}
		\item Sample $\hat{\mu}_1^* \sim \text{Multinomial}(n_1,\widehat{\mu}_1)$.
		\item For each $i=1,\hdots,k$, \\
		sample $\hat{G}_i^* \sim \text{Normal}(0,\text{diag}(\widehat{\mu}_1^*) - \hat{\mu}_1^* (\hat{\mu}_1^*)')$. \\
		Let $a_i = \prod_{j \neq i} \hat{\lambda}_j$, where $\hat{\lambda}_j = \frac{n_j}{n_1+\hdots+n_k}$.
		\item Solve the program \eqref{eq:UB_0} with $\{\hat{G}_i^*\}_{i=1}^k$ in place of $\{G_i\}_{i=1}^k$. 
	\end{enumerate}
\end{pseudocode}

Computational complexity of sampling $UB_0$ is included in Table \ref{table:complexity}. The performance of $UB_0$ for testing $H_0$ on all real datasets considered in the paper is illustrated in Figure \ref{fig:ub}. Note that the low computational complexity of $UB_0$ allows to approximate the $H_0$ distribution on large datasets within a few minutes on a standard laptop.

\subsection{Permutation approach}\label{sec:perm}
An alternative to the asymptotic test \eqref{eq:test} is a permutation test. The permutation approach in k-sample testing is frequently used when the asymptotic distribution is difficult to sample from due to, for example, infinite number of parameters and/or difficulties of their estimation (cases of \cite{Rizzo}, \cite{Huskova}, and \cite{Kim}). Moreover, permutation procedures are applicable when the sample sizes are small (and hence the asymptotic distribution may not be valid), giving exact level $\alpha$ permutation tests (Section 15.2 of \cite{Efron}).  

Permutation test accepts a set of data points with group labels, and randomly permutes the labels to compute test statistic of interest on the permuted data to compare with the original one. The number of random permutations $R$ is usually taken to be between $99$ and $999$ out of total possible large number of permutations (p. 158 of \cite{Davison}). 

The $MOT$ permutation test is described in Pseudocode \ref{pseudocode:perm}. 
\begin{pseudocode}[\textbf{MOT based permutation test}]\label{pseudocode:perm}
	Given the data $\widehat{\mu}_n=(\widehat{\mu}_1,\cdots, \widehat{\mu}_k)$: 
	\begin{enumerate}
		\item  Compute $MOT(\widehat{\mu}_n)$. 
		\item Convert $\widehat{\mu}_n$ to a matrix of support points, where each support point belongs to the $i$th group, $i=1,\hdots,k$, and is repeated according to the counts in $\widehat{\mu}_i$. Collect group labels in the vector $v$.
		\item For each $r=1,\hdots,R$, sample random permutation $\pi_r(v)$, permute support points according to $\pi_r(v)$, and construct measures $\widehat{\mu}_n^r$ based on the frequencies of support points in new groups. Compute permuted test statistic $MOT(\widehat{\mu}_n^r)$ by solving the program \eqref{eq:mot_primal}.
		\item Compute approximate p-value (p. 158 of \cite{Davison}) as
		$$\hat{p}:=\frac{1+\sum_{r=1}^R \mathbbm{1}_{\{MOT(\widehat{\mu}_n^r) \geq MOT(\hat{\widehat{\mu}}_n)\}}}{1+R}$$
	\end{enumerate}
\end{pseudocode}

Computation of permuted test statistic in Step 3 requires to solve the $MOT$ program \eqref{eq:mot_primal}, similarly to the case of m-out-of-n bootstrap in Pseudocode \ref{pseudo:mn}. Hence, both algorithms have the same complexity, as shown in Table \ref{table:complexity}. Empirical performance of the permutation test of Pseudocode \ref{pseudocode:perm} is illustrated in Figure \ref{fig:synth_3D}. 


\begin{table}
	\caption{Complexity of computing a single sample for permutation null distribution, $X_0$ with m-out-of-n bootstrap, $X_0$ with derivative bootstrap, and ${UB}_0$ given by Theorem \ref{thm:as_dist}(b). Recall that $k$ is the number of measures $\mu_1,\hdots, \mu_k$, and $N$ is the cardinality of the underlying metric space $\mathcal{X}=\{x_1,\hdots,x_N \} \subset \mathbb{R}^d$. Algorithm (theory) row reports an algorithm used to prove theoretical complexity, while algorithm (practice) rows report algorithms implemented in this paper and available for use. The algorithm of Altschuler \& Boix-Adser{\`a} \cite{Altschuler} is abbreviated as AB-A and assumes fixed $d$.}
	\begin{tabular}{c||c|c|c}
		distribution to sample & permut. or $X_0$ (m-out-of-n) &
		$X_0$ (deriv.)
		& $UB_0$  \\
		hypothesis & null, alternative& null& null\\
		optimization program & equation \eqref{eq:mot_primal} & equation \eqref{eq:limdist_null_expression} & equation \eqref{eq:UB_0} \\
		\hline \hline
		$\#$ variables &$N^k$ & $kN$ & $(k-1)N$\\
		$\#$ equality constraints& $kN$ & $N$ & none \\
		$\#$ inequality constraints& none & $N^k$& $(k-1)N(N-1)$ \\
		\hline \hline
		theoretical complexity & $\text{poly}(N,k,\log U)$  & $\text{poly}(N,k,\log U)$ & $\text{poly}(N,k,\log U)$ \\
		
		reference &Theorem 2 of \cite{Altschuler}& Lemma \ref{lem:der_boot_compl} here & Theorem 6 of \cite{Altschuler} \\
		algorithm (theory) & AB-A \cite{Altschuler} &  AB-A \cite{Altschuler} & Ellipsoid \\
		\hline
		\multirow{3}{4em}{algorithm (practice)} & AB-A \cite{Altschuler} & AB-A \cite{Altschuler} & \\ 
		& Simplex & Simplex & Simplex \\
		& Interior point & Interior point & Interior point\\ 
		\hline
		\multirow{3}{4em}{software} & Github for \cite{Altschuler} (d=2) & & \\ 
		& GUROBI \cite{gurobi}& GUROBI \cite{gurobi} & GUROBI \cite{gurobi} \\
		& RSymphony \cite{Rsymphony} &RSymphony \cite{Rsymphony} &RSymphony \cite{Rsymphony} \\
	\end{tabular}
	\label{table:complexity}
\end{table}

\section{Applications}\label{sec:applications}
Sections \ref{subsec:synth} and \ref{subsec:maps} illustrate basic properties of $MOT$ based inference on synthetic datasets with measure supports on finite subsets of $\mathbb{R}^d$, $d=1,2,3$. The structures of these datasets emulate potential issues in the real data settings while providing convenient models to demonstrate the advantages of $MOT$ based procedures over existing methods (Figures \ref{fig:synth_3D} and \ref{fig:maps}).

Sections \ref{subsec:real_1D} and \ref{subsec:real_year} illustrate how $MOT$ based inference can be used in real biomedical settings where measures of interest are naturally finitely supported on a given metric space. We use \textit{Surveillance, Epidemiology, and End Results (SEER)} \cite{SEER}, a large database on cancers in the United States routinely used in biomedical literature. Detailed description of the used data including the information on the sample sizes is provided in Appendices \ref{appB:size} and \ref{appB:year}.

\subsection{Illustrations on synthetic data}\label{sec:illustration_synth}
\subsubsection{\textbf{3D Experiment} dataset: testing $H_0$}\label{subsec:synth}

We construct the dataset \textbf{3D Experiment} which aims to emulate experimental settings of counting the number of induced cells in response to a treatment. The model organism frequently used in such experiments is the nematode worm \textit{C. elegans}. The goal of the experiment is to determine whether certain genetic modification interrupts with a normal organ development, resulting in abnormal cell behavior observed in diseases \cite{Corchado, Zand}.

The abnormality is measured by the number of induced cells that emerge after genetic modification. There could be $0,1,2$ induced cells in each worm; a total of $n$ worms are examined giving a measure supported on $\{0,1,2\}$\footnote{In the real experiment, the measure is supported on $\{0,1,2,3 \}$, but we simplify it for the purpose of this synthetic dataset}. When constructing \textbf{3D Experiment} dataset, we assume that counting is simultaneously performed in two more sites of an animal, where the number of induced cells can be $\{0,1\}$. This results in the 3-dimensional support $\{0,1\}\times\{0,1\}\times \{0,1,2\}$ with $2\times 2\times 3 = 12$ points (Figure \ref{fig:synth_3D}A). 

Recent results in biological literature report differences in the number of induced cells between worm species \textit{C. elegans} and \textit{C. briggsae} \cite{Dawes, Mahalak}. 
Inspired by these results, we construct four measures $\{\mu_1,\mu_2,\mu_3,\mu_4\}$ in \textbf{3D Experiment} dataset: the first two correspond to two \textit{C. briggsae} worm strains, and the last two \textit{C. elegans} worm strains (Figure \ref{fig:synth_3D}A). We use this set up to demonstrate the power of MOT asymptotic and permutation tests for testing $H_0$ (Figure \ref{fig:synth_3D}B).

\begin{figure}
	\includegraphics[width=1\textwidth]{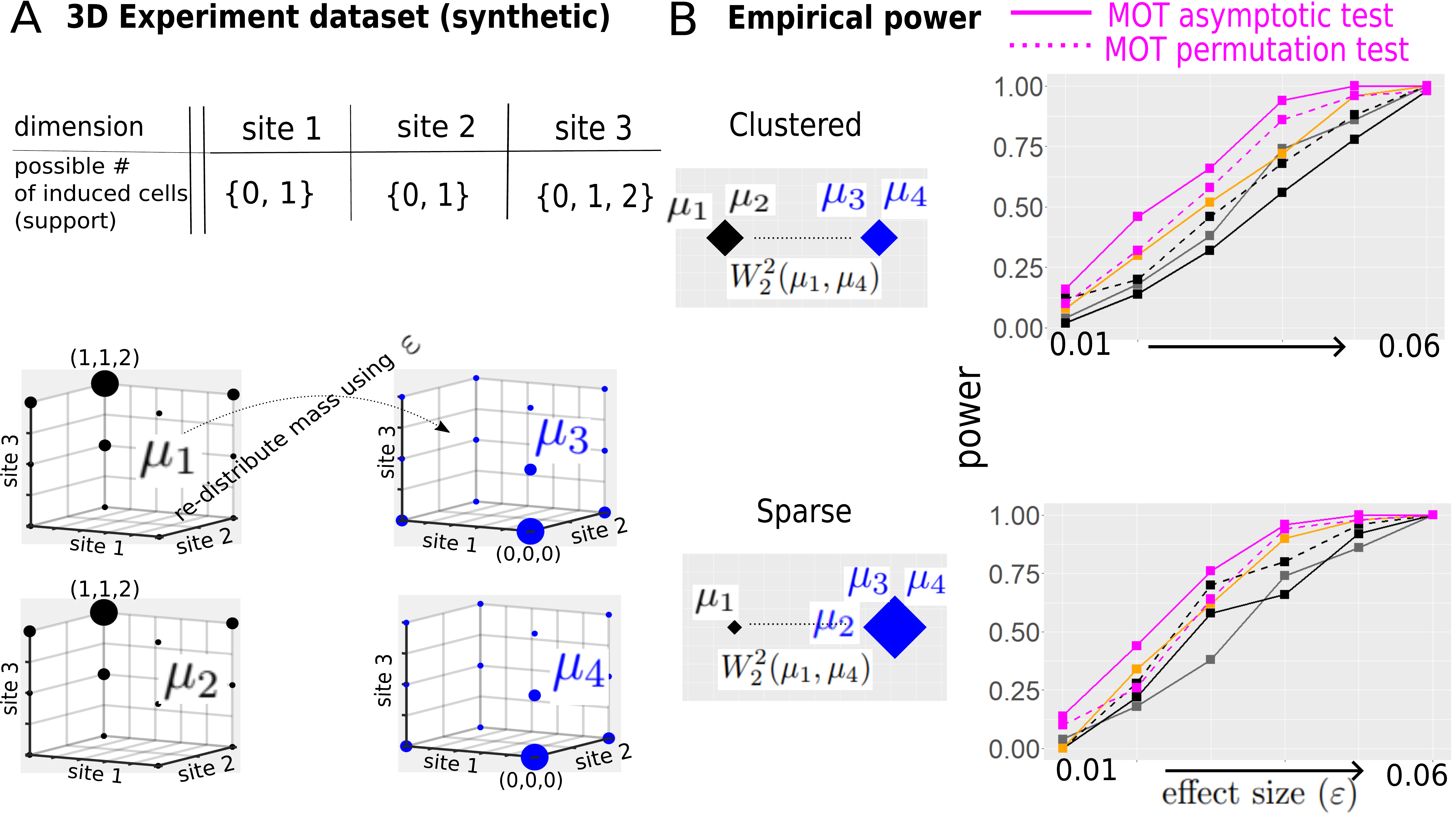}
	\caption{$MOT$ based testing of $H_0$ and empirical power on \textbf{3D Expeiment} synthetic data. A. Structure and plot of the true measures (``clustered" alternative is shown). B. Left: Illustration for ``clustered" and ``sparse" alternatives discussed in Section \ref{subsec:power}. Right:  Empirical power against ``clustered" $(\mu_1=\mu_2\neq\mu_3=\mu_4)$ and ``sparse" $(\mu_1=\mu_2=\mu_3\neq \mu_4)$ alternatives. MOT asymptotic test \eqref{eq:test} and MOT permutation test (section \ref{sec:perm}) are compared with the following four tests: distance based test DISCO from \cite{Rizzo} (gray), kernel based tests from \cite{Kim} with Gaussian (solid black) and energy distance (dashed black) kernels, and test based on empirical characteristic functions from \cite{Huskova} (orange). Sample sizes are taken as $n=300$ and $n=500$ for two classes, respectively.} 
	\label{fig:synth_3D}
\end{figure}
\subsubsection{\textbf{Anderes et al. 2016} dataset: $H_a$ inference} \label{subsec:maps}

We consider the data constructed by \cite{Anderes}, where it demonstrates the properties of a barycenter of finitely supported measures. Each measure represents a demand distribution (for some hypothetical product) over nine locations on the map (these locations are cities in California, and they constitute a finite support $\mathcal{X}=\{x_1,\hdots, x_9\} \subset \mathbb{R}^2$ for demand distributions). There are $12$ measures in \cite{Anderes}, each giving demand distribution during particular month (Figure \ref{fig:maps}A,B).

We use these measures as the ground truth and construct empirical measures by sampling multinomial counts based on this truth. We note that all $12$ underlying true measures are different, i. e., $H_a$ holds. Moreover, the differences are more drastic between months with different temperature since \cite{Anderes} allows the temperature to influence the demand. Our inference under $H_a$ confirms this claim by examining sub-collections of measures with months from the same season versus months from different seasons and comparing Confidence Regions for $MOT$ under these settings (Figure \ref{fig:maps}C). 

\begin{figure}
	\includegraphics[width=1\textwidth]{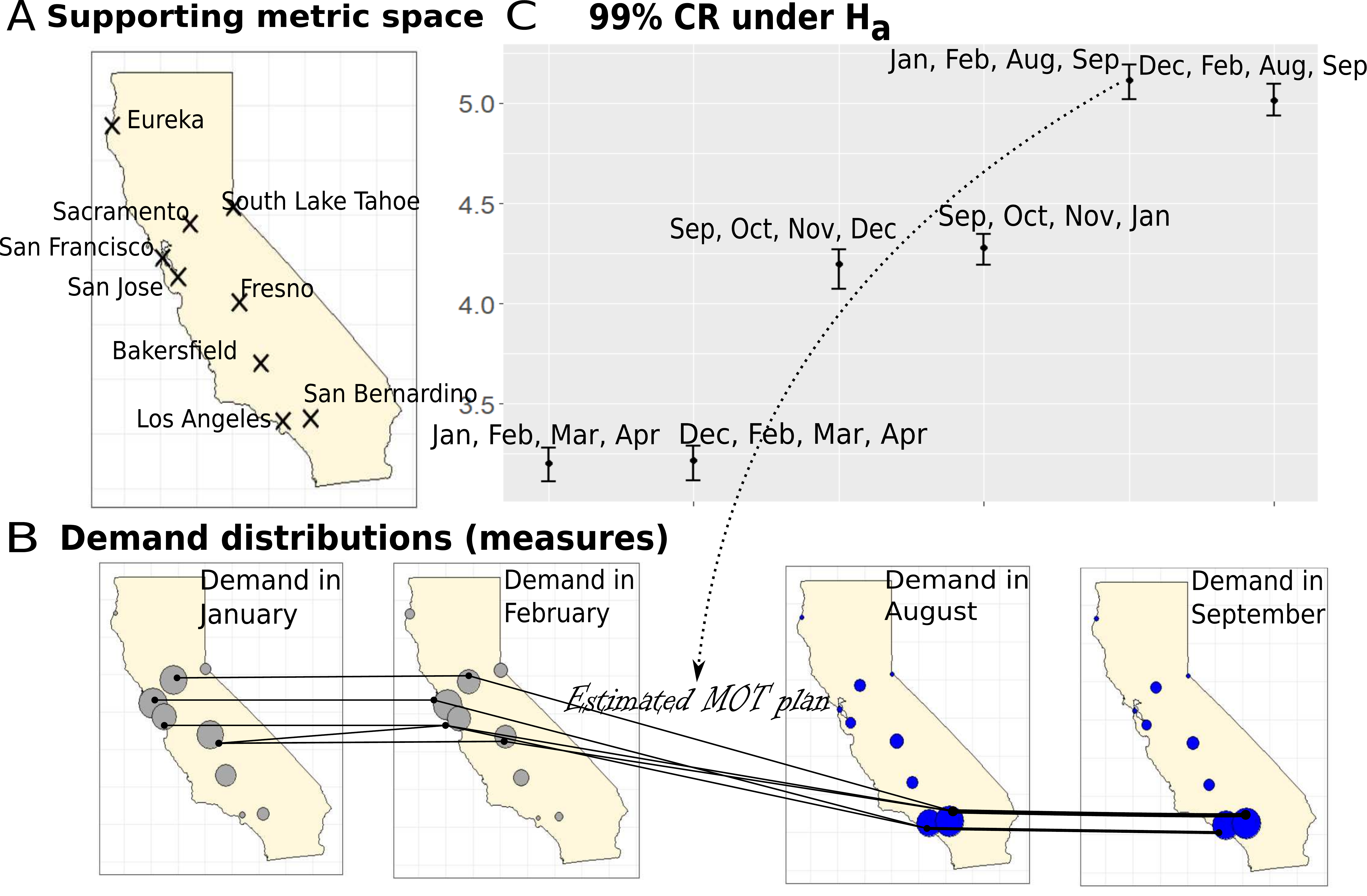}
	\caption{$MOT$ based inference under $H_a$ for the \textbf{Anderes et al. 2016} dataset from \cite{Anderes} described in Section \ref{subsec:maps}. A. Schematic of the state of California map with nine cities supporting the $12$ measures corresponding to monthly demand distributions (four of them shown in B). B. Illustration of estimated $MOT$ plan (multicoupling; five highest probability non-identity tuples shown) coupling the support points of four measures. C. $99 \%$ Confidence Regions for $MOT$ cost under $H_a$ for 4-measure collections. Note an overlap of similar collections and no overlap between different collections, as well as higher $MOT$ costs for collections whose monthly demands are more different. }
	\label{fig:maps}
\end{figure}

\subsection{Applications to real data}\label{sec:application_real}
\subsubsection{\textbf{SEER Tumor size} dataset: testing $H_0$} \label{subsec:real_1D} 
An important question in cancer studies is to determine what factors are associated with development of metastases. In the case of breast cancer,  \cite{Sopik} showed that metastatic risk increases with tumor size in intermediate and some of the large tumors ($\geq 1$ cm), but does not increase in small tumors ($<1$ cm). The study used SEER database and considered a correlation between tumor size and prevalence of metastases. Here we confirm these results via k-sample testing, as described below.  Further, we observe similar trend in three more cancer types: prostate cancer, lung cancer in males, and lung cancer in females (Figure \ref{fig:sizes_1D}).  

We use the SEER database to extract the data on distributions of tumor size and term this dataset \textbf{SEER Tumor size}. We consider three groups of patients with different disease progression status, giving $k=3$ measures: patients with no metastases present at diagnosis and alive at the end of the study ($\mu_1$), patients with metastases at diagnosis and alive at the end of the study ($\mu_2$), and patients dead by the end of the study with death caused by the diagnosed cancer ($\mu_3$). 

First, we test $H_0$ for size distributions in small tumor range ($<1$ cm); we find no difference between groups, which holds for breast, prostate, and both lung cancer types (Figure \ref{fig:sizes_1D}A). In contrast, the groups are found different for tumors in larger range ($1 - 9$ cm), which again holds for all considered cancer types (Figure \ref{fig:sizes_1D}B).
The analysis confirms the significance of metastatic status for the tumor size distribution in intermediate/large tumors, but not the small tumors. 

\begin{figure}
	\includegraphics[width=0.95\textwidth]{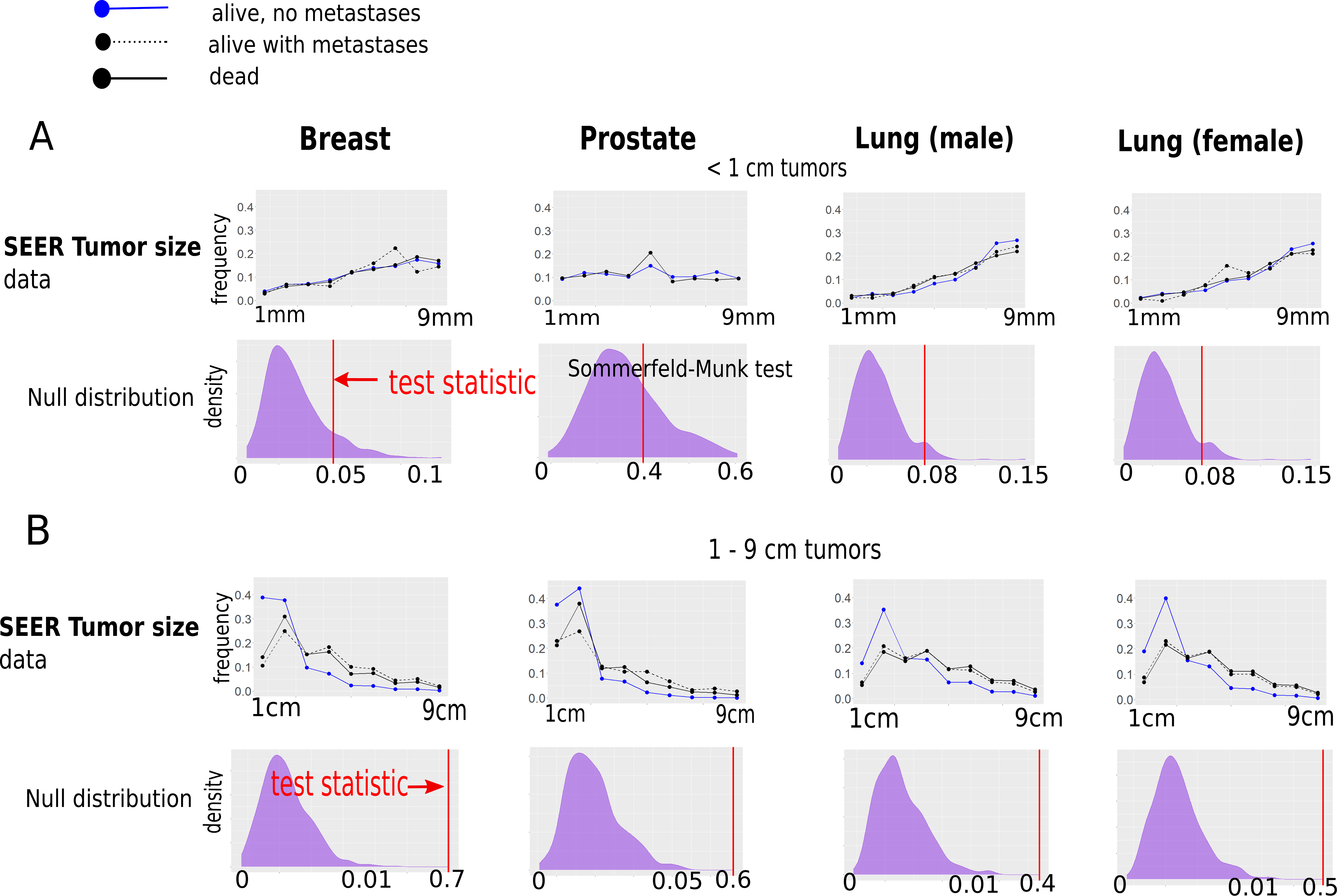}
	\caption{Application of $MOT$ based testing of $H_0$ to comparison of tumor size distributions from \textbf{SEER Tumor size} dataset described in Section \ref{subsec:real_1D}. Size distributions are compared for three groups of patients with different metastatic/survival characteristics (alive at the end of the study with no metastases at diagnosis, alive with metastases present at diagnosis, dead at the end of the study). All sample sizes are very large, except the prostate cancer case (Appendix \ref{appB:size}). A. Comparison of sizes in smaller tumor size range shows no difference (at $1 \%$ level) among three groups of patients. B. Comparison of sizes for larger tumors shows significant difference (at $1 \%$ level) between three groups of patients. \underline{Note}: For the prostate cancer with small tumor sizes, enough data was available only for two groups of patients rather than three, so we used the Sommerfeld-Munk test \cite{Sommerfeld} in place of the $MOT$ test and obtained similar conclusion. }
	\label{fig:sizes_1D}
\end{figure}

\subsubsection{\textbf{SEER Year of diagnosis} dataset: $H_a$ inference} \label{subsec:real_year} 
Our final example concerns with potential differences in distributions of characteristics in patients diagnosed at different times. Such differences are discussed in the case of early stage lung cancer, possibly due to improvements in diagnostic technologies \cite{Nations}. Here we compare these distributions in a framework of $k$-sample inference to confirm the differences in diagnosis results over time, and show that the trend is similar in both male and female patients.

We use SEER database to extract joint distributions of tumor size and patients' age for lung cancer in males and females and term this dataset \textbf{SEER Year of diagnosis}. We consider four time periods giving $k=4$ measures: 2004 - 2006 ($\mu_1$), 2009 - 2011 ($\mu_2$), 2014 - 2016 ($\mu_3$), 2019 - 2020 ($\mu_4$). The distributions are found different by $MOT$ test in both male and female lung cancer cases, and we are interested to compare the differences between male and female collections of measures (Figure \ref{fig:year}).

We observe visually that the differences between measures are of similar nature in male and female cases: later diagnostic years appear to have more small size tumors diagnosed in comparison to earlier years (Figure \ref{fig:year}A). The similarities between male and female cases are reflected in overlapping Confidence Regions. The reported $MOT$ plan also highlights this finding by coupling the small size support points from the later periods with the larger size support points from the earlier period for patients of the same age (Figure \ref{fig:year}B).


\begin{figure}
	\includegraphics[width=0.95\textwidth]{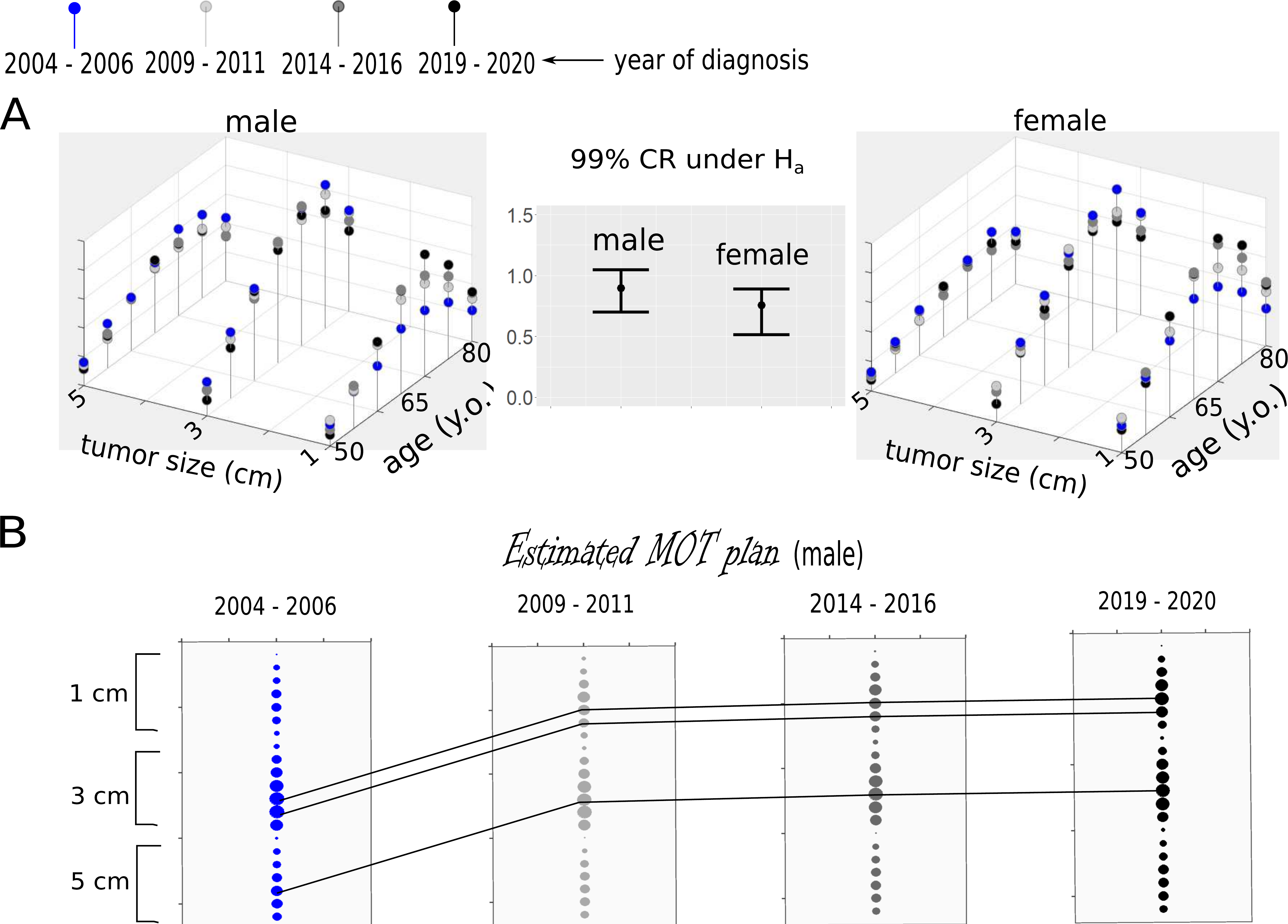}
	\caption{Application of $MOT$ based inference under $H_a$ to comparison of bivariate (age/tumor size) distributions from \textbf{SEER Year of diagnosis} dataset described in Section \ref{subsec:real_year}. Bivariate distributions are compared for four periods of diagnosis: 2004 - 2006, 2009 - 2011, 2014 - 2016, and 2019 - 2020 In both male and female cases, the age/tumor size distributions are different between four periods of diagnosis (i.e., $H_0$ is rejected). A. Confidence Regions for the $MOT$ cost for male and female cases overlap, suggesting that differences between distributions are of similar magnitude. B. Illustration of estimated $MOT$ plan (multicoupling) for male case (non-identity tuples with positive multicoupling mass are shown).}
	\label{fig:year}
\end{figure}

\section{Discussion and Conclusions}\label{sec:conclusion}
\subsection{Summary of results}\label{sec:summary}
In this paper, we proposed an Optimal Transport approach to $k$-sample inference. We used the optimal value of the Multimarginal Optimal Transport program ($MOT$) to quantify the difference in a given collection of $k$ measures supported on finite subsets of $\mathbb{R}^d$, $d\geq 1$. 

We derived limit laws for the empirical version of $MOT$ under assumptions of $H_0$ (all $k$ measures are the same) and $H_a$ (some measures may differ). We established that the limit cannot be Gaussian under $H_0$, and provided sufficient conditions for the limit to be Gaussian under $H_a$. Based on these results, we derived expression for the power function of the test of $H_0$; using this function, we proved consistency of the test against any fixed alternative and uniform consistency in certain broad classes of alternatives. 

To sample from limit laws, we confirmed that derivative and m-out-of-n bootstrap methods are consistent under $H_0$, and m-out-of-n bootstrap is consistent under $H_a$. We proved polynomial complexity of sampling via derivative bootstrap, and defined a low complexity upper bound to approximate the test cut-off under $H_0$. As an alternative to sampling for the limit laws, we defined a permutation test that is suitable if sample sizes are not large enough to validate to convergence to the limit. 

We empirically showed that the $MOT$ based test of $H_0$ has strong finite sample power performance when compared with state-of-the-art methods. We also showed how to construct Confidence Regions for the true $MOT$ value under the assumptions of $H_a$, and how to use this procedure to compare variability between collections of $k$ measures. Finally, we demonstrated the use of our methodology on several real biomedical datasets.

\subsection{Limitations and future directions}\label{sec:limitations}
\par{\textit{Extensions to continuous measures:}} One of the main benefits of working on finite spaces is the ability to obtain a non-degenerate limit law under $H_0$ (i. e. the law with a non-zero variance),  which allows to quantify fluctuations of the $MOT$ value when all measures are the same and test $H_0$. In $k=2$ case, non-degeneracy may fail for continuous measures (see discussion in Section \ref{sec:weak_limits}), but holds for discrete measures with limit laws of the form \cite{Sommerfeld}. When extending \cite{Sommerfeld} to continuous measures in $k=2$ case, \cite{Hundrieser} show that non-degenerate limit laws are possible provided that there exist dual variables (the \textit{Kantorovich potentials}) which are not constant almost everywhere (Theorem 4.2 of \cite{Hundrieser}). While constant potentials are always present under $H_0$ (Corollary 4.6 of \cite{Hundrieser}), in discrete case there are other potentials around that are not constant (this holds for our case of $k>2$, see discussion preceding Theorem \ref{thm:as_dist}). Lemma 11 of \cite{Staudt} shows that in $k=2$ case, it is possible to get non-constant potentials for continuous measures by requiring the support to be disconnected (intuitively, this resembles the discrete situation). It is an interesting future direction to analyze how the $k>2$ potentials would behave if our limit laws are extended to continuous measures and possibly different ground cost $c$.

\par{\textit{Improving upper bounds on the null distribution:}} While the proposed null upper bound $UB_0$ is computationally tractable and tight for $k=2$ measures, it may be too large to provide a good power for $H_0$ testing for larger $k$. The main reason for this weakness is the ``independent" nature of optimization over dual variables $u_1,\hdots,u_k$ recorded in the constraints. Indeed, the constraints that relate different entries of different dual vectors are omitted, and hence the dual vectors only interact via $\sum_{i=1}^ku_i = 0$ when solving the $UB_0$ program. The bound $UB_0$ can be strengthened by introducing additional constraints from $X_0$, which will decrease the value of the program and provide a tighter bound on the null distribution. Two possible choices for these extra constraints are (1) including constraints that involve diverse entries from different dual vectors (e.g., $u_1^1+u_2^2+\cdots+u_k^k \leq c_{12\cdots k}$), and/or (2) sampling constraints at random (such constraint sampling techniques are widely applicable when solving large linear programs arising, for instance, in Markov Decision Processes \cite{Farias}). 

\par{\textit{Faster computation of $MOT$/barycenter value and permutation test:}} Our empirical power results suggest that $MOT$ could serve as a suitable statistic for a powerful permutation test (for example, when sample sizes are not large enough to validate an asymptotic test). In that case, the $MOT$ (or, equivalently, the barycenter) value has to be computed for each permutation. While computation of the $MOT$/barycenter value is challenging for a large cardinality of the support $N$ and/or number of measures $k$, recently proposed subsampling techniques \cite{Heinemann} and algorithmic tools \cite{FriescekePenka} can be used to keep up with speed and memory requirements.

\begin{appendix}
	\section{Proofs of main results omitted in the main text}\label{appA}
	\subsection{Details on the proof of Theorem \ref{thm:as_dist}(a)}\label{app:proof_thm_a}
	\underline{Step 1}  (Weak convergence of measures) For every $i=1,\hdots,k$, the empirical process converges weakly (Theorem 14.3 - 4 of \cite{Bishop}) as
	$$\sqrt{n_i}\left( \widehat{\mu}_i - \mu  \right) \xrightarrow[]{\text{in law}} G_i$$
	where $G_i \sim N(0, \text{diag}(\mu_i - \mu_i\mu_i') )$. Since the processes are independent in $i=1,\hdots,k$, by Theorem 1.4.8 of \cite{vanderVaartbook} we can view them jointly as
	
	\begin{equation*}
		\begin{pmatrix}\sqrt{n_1}(\hat{\mu}_1 - \mu_1) \\ \vdots \\ \sqrt{n_k}(\hat{\mu}_k - \mu_k) \end{pmatrix} \xrightarrow[]{\text{in law}} \begin{pmatrix} G_1 \\ \vdots \\ G_k \end{pmatrix} 
	\end{equation*}
	with respect to $l^1$ norm on $\bigotimes_{i=1}^k l^1(\mathcal{X})$. Using Slutsky's Theorem (e.g., Example 1.4.7 of \cite{vanderVaartbook})
	\begin{equation*}
		\begin{pmatrix}\sqrt{n_1}\frac{\sqrt{n_2}}{\sqrt{n_1+ \hdots + n_k}} \cdots \frac{\sqrt{n_k}}{\sqrt{n_1+ \hdots + n_k}}(\hat{\mu}_1 - \mu_1) \\ \vdots \\ \sqrt{n_k} \frac{\sqrt{n_1}}{\sqrt{n_1+ \hdots + n_k}} \cdots \frac{\sqrt{n_{k-1}}}{\sqrt{n_1+ \hdots + n_k}} (\hat{\mu}_k - \mu_k) \end{pmatrix} \xrightarrow[]{\text{in law}} \begin{pmatrix} \overbrace{\sqrt{\lambda_2} \cdots \sqrt{\lambda_k}}^{\sqrt{a_1}}G_1 \\ \vdots \\ \underbrace{\sqrt{\lambda_1} \cdots \sqrt{\lambda_{k-1}}}_{\sqrt{a_k}} G_k \end{pmatrix} =:G
	\end{equation*}
	which is of the form
	\begin{equation*}\rho_n (\widehat{\mu}_n - \mu) \xrightarrow[]{\text{in law}} \sqrt{a} G
	\end{equation*}
	with $\rho_n:= \frac{\sqrt{n_1 \cdot \hdots \cdot n_k}}{\left(\sqrt{n_1+\hdots+n_k}\right)^{k-1}}$. 
	
	\underline{Step 2} (Hadamard directional differentiability of $MOT$) Consider the functional $f: \bigotimes_{i=1}^k \mathcal{P} (\mathcal{X}) \subseteq \bigotimes_{i=1}^k l^1(\mathcal{X})   \longrightarrow \mathbb{R}$ given by $f(\mu) = MOT(\mu)$, where $MOT(\mu)$ is the optimal value $z$ of the primal program \eqref{eq:mot_primal}, or, equivalently, the dual program \eqref{eq:mot_dual}. The map $\mu \longrightarrow z(\mu)$ is G{\^a}teaux directionally differentiable at $\mu$ tangentially to a certain set $D \subseteq \bigotimes_{i=1}^k \mathcal{P}(\mathcal{X})$ (Theorem 3.1 of \cite{Gal}) and locally Lipschitz (Remark 2.1 of \cite{Carcamo}\footnote{Locally Lipschitz property can also be shown directly (Appendix \ref{appB}).}). Hence, it is Hadamard directionally differentiable at $\mu$ tangentially to $D$ and two derivatives coincide (see Proposition 3.5 of \cite{Shapiro} and also the discussion in Section 2.1 of \cite{Carcamo}). The derivative is given by
	\begin{equation}\label{eq:gatoux_der}
		f'_\mu(g) =  \max_{\{u: A'u \leq c, \langle u,\mu  \rangle = MOT(\mu)\}} \langle u, g\rangle
	\end{equation}
	for $g \in D:=\{\lim_{n \to \infty}\frac{\widehat{\mu}_n - \mu}{t_n}$, $\widehat{\mu}_n \in l^1(\mathcal{X})$, $t_n \searrow 0\}$.


	\subsection{Details on the proof of Theorem \ref{thm:as_dist}(b)}\label{app:proof_thm_b}
	Here we construct an upper bound ${UB}_0$ on the null random variable $X_0$ with lower computational complexity than $X_0$ by relaxing some constraints of the $X_0$ program. 
	Consider the null distribution of $X_0 \sim \mathcal{D}_0$ in \eqref{eq:limdist_null_expression} given by 
	\begin{equation*}
		\begin{aligned}
			\max_{u} & \sum_{i=1}^k \sqrt{a_i}\langle u_i, G_i   \rangle \\
			\text{ s.t. } & \sum_{i=1}^k u_i = 0 \\
			& A'u \leq c
		\end{aligned}
	\end{equation*}
	Note that for any given realization of random coefficients $G_i\sim\mathcal{N}\left(0, \Sigma_1 \right)$, this linear program has $k$ dual vectors $\{u_i\}_{i=1}^k$, each containing $N$ entries  (which gives $kN$ variables in total). There are $N$ equality constraints in $\sum_{i=1}^k u_i$, each corresponding to summation of a particular entry of these dual vectors to zero. A large number $N^k$ of inequality constraints in $A'u \leq c$ comes from the size of the primal constraint matrix $A \in \mathbb{R}^{kN\times N^k}$ (the structure of this matrix is discussed in Section \ref{sec:prelim}).
	
	To construct ${UB}_0 \sim \mathcal{D}_{{UB}_0}$ with smaller complexity than $X_0$, we take the same objective function as in $X_0$ program above, but relax some of the inequality constraints $A'u\leq c$ subject to equality constraints $\sum_{i=1}^k u_i = 0$. Formally, we represent the equality constraints of $X_0$ as
	$$\sum_{i=1}^k u_i = 0 \iff u=\begin{bmatrix} u_1 \\ \vdots \\ u_k  \end{bmatrix} \in \text{ker}(B)$$
	where $B$ represents the linear operator with matrix whose $j${th} row picks $j${th} element from vectors $u_1,\hdots, u_k$ and sums them up. 
	
	As detailed in the proof of Lemma \ref{lem:dual_bd_null} given in \ref{app:proof_dual_bd}, for $u \in \text{ker}(B)$, there are constraints in $A'u\leq c$ of the form
	$${u_1}_1 - {u_1}_j \leq c_{1\cdots 1 j}$$
	which can be written as
	$$\underbrace{\begin{pmatrix} 1 & -1 & 0 & \cdots & 0 \\ 1 & 0 & -1 & \hdots & 0\\ \vdots &  \vdots &\vdots&  \ddots & \vdots \\1 & 0 & 0 & \cdots & -1\end{pmatrix}}_{\widetilde{A}'_1} u_1 \leq \underbrace{\begin{pmatrix}c_{1\cdots 1 2} \\ \vdots \\ \vdots \\ c_{1\cdots 1 N} \end{pmatrix}}_{\widetilde{c}_1}$$
	Similarly, for $u \in \text{ker}(B)\cap\{A'u\leq c \}$, the first dual vector $u_1$ satisfies 
	$$\underbrace{\begin{pmatrix} -1 & 1 & 0 & \cdots & 0 \\0 & 1 & -1 & \hdots & 0\\ \vdots &  \vdots &\vdots&  \ddots & \vdots \\0 & 1 & 0 & \cdots & -1\end{pmatrix}}_{\widetilde{A}'_2} u_1 \leq \underbrace{\begin{pmatrix}c_{2\cdots 2 1} \\ \vdots \\ \vdots \\ c_{2\cdots 2 N} \end{pmatrix}}_{\widetilde{c}_2}$$
	$$\vdots$$
	$$\underbrace{\begin{pmatrix} -1 & 0 & 0 & \cdots & 1 \\0 & -1 & 0 & \hdots & 1\\ \vdots &  \vdots &\ddots&  \vdots & \vdots \\0 & 0 & 0 & -1 & 1\end{pmatrix}}_{\widetilde{A}'_N} u_1 \leq \underbrace{\begin{pmatrix}c_{2\cdots 2 1} \\ \vdots \\ \vdots \\ c_{N\cdots N N-1} \end{pmatrix}}_{\widetilde{c}_N}$$
	Thus, for $u \in \text{ker}(B)\cap\{A'u\leq c \}$,
	$$\begin{pmatrix}\widetilde{A}_1' \\ \vdots \\ \widetilde{A}_N'  \end{pmatrix} u_1 \leq \begin{pmatrix} \widetilde{c}_1 \\ \vdots \\ \widetilde{c}_N \end{pmatrix} $$
	and the same constraints are satisfied by $u_2,\hdots,u_k$. Combining these constraints, we obtain
	$$\underbrace{\begin{pmatrix} {\begin{pmatrix}\widetilde{A}_1' \\ \vdots \\ \widetilde{A}_N'  \end{pmatrix}} & 0 & \hdots & 0 \\ 0 & {\begin{pmatrix}\widetilde{A}_1' \\ \vdots \\ \widetilde{A}_N'  \end{pmatrix}} & \hdots & 0 \\ \vdots & \vdots & \ddots & \vdots \\ 0& 0& \hdots & {\begin{pmatrix}\widetilde{A}_1' \\ \vdots \\ \widetilde{A}_N'  \end{pmatrix}} \end{pmatrix}}_{\widetilde{A}'} \begin{bmatrix} u_1 \\ \vdots \\ \vdots \\ u_k \end{bmatrix} \leq \underbrace{\begin{pmatrix} {\begin{pmatrix} \widetilde{c}_1 \\ \vdots \\ \widetilde{c}_N \end{pmatrix}} \\ \vdots \\ \vdots \\ {\begin{pmatrix} \widetilde{c}_1 \\ \vdots \\ \widetilde{c}_N \end{pmatrix}}\end{pmatrix}}_{\widetilde{c}}$$
	for all $u \in \text{ker}(B)\cap\{A'u\leq c \}$.
	This gives us $\widetilde{A}'u\leq \widetilde{c}$ with $kN(N-1)$ constraints that we choose to be the constraint set for the linear program \eqref{eq:UB_0} defining ${UB}_0$ (which now has no equality constraints).
	
	Note that for $u \in \text{ker}(B)$, we have that $u_1 = -\sum_{i=2}^k u_i$, which, upon substitution to the objective function of \eqref{eq:limdist_null_expression}, gives
	\begin{equation*}
		\begin{aligned} \langle u_2, \sqrt{a_2} G_2  \rangle + \cdots + \langle u_k, \sqrt{a_k} G_k  \rangle & - \langle u_2,\sqrt{a_1} G_1  \rangle - \cdots - \langle u_k,\sqrt{a_1} G_1 \rangle \\
			& \overset{d}{=} \sum_{i=2}^k \langle u_i, \sqrt{a_i}G_i - \sqrt{a_1}G_1  \rangle
		\end{aligned}
	\end{equation*}
	This is the objective in the linear program \eqref{eq:UB_0} defining ${UB}_0$ with $(k-1)N$ variables. 
	
	\subsection{Details for Observation \ref{obs:cost_dist}}\label{app:proof_cost_dist}
	Consider a cost vector in the $MOT$ program \eqref{eq:mot_primal}  with entries $c_{i_1i_2\cdots i_k}$, where each index takes values in $\{1,\hdots,N \}$. Suppose that $k-1$ indexes have the same value, e.g. $c_{i \cdots i j}$. Then,
	
	\begin{equation*}
		\begin{aligned}
			c_{i\cdots i j} = \frac{1}{k} \left( (k-1)\left\|x_i - \overline{x}_{i\cdots ij} \right\|^2 + \left\|x_j -  \overline{x}_{i\cdots ij} \right\|^2 \right)
		\end{aligned}
	\end{equation*}
	where $\overline{x}_{i\cdots ij} = \frac{1}{k}\left[ (k-1)x_i + x_j  \right]$. The first term gives 
	$$ (k-1)\left\|x_i - \overline{x}_{i\cdots ij} \right\|^2 = \frac{k-1}{k^2} \|x_i - x_j \|^2$$
	and the second term gives
	$$\left\|x_j -  \overline{x}_{i\cdots ij} \right\|^2 = \frac{(k-1)^2}{k^2} \|x_i - x_j \|^2$$
	Combining the two and multiplying by $\frac{1}{k}$ gives the result.
	\subsection{Proof of Lemma \ref{lem:dual_bd_null}} \label{app:proof_dual_bd}
	For notational clarity, we start with a case of $k=2$ measures, and assume for simplicity that they are supported on the metric space $\mathcal{X}$ with only two points. Let $u=\left((u_1,u_2)=({u_1}_1 {u_1}_2), ({u_2}_1,{u_2}_2) \right)$ be solutions to the dual $MOT$ program \eqref{eq:mot_dual} satisfying $\sum_{i=1}^k u_i = 0$, i.e. $u_2=-u_1$.
	Recall that the constraint matrix in the dual constraints $A'u\leq c$ is
	$$A'=\begin{pmatrix} 1&0&1&0\\1&0&0&1\\0&1&1&0\\0&1&0&1 \end{pmatrix}$$
	Applying it to $u=\begin{pmatrix} u_1 \\ -u_1 \end{pmatrix}$ gives the constraints on $u_1$ as
	\begin{equation*}
		\begin{aligned}
			& {u_1}_1 - {u_1}_1 \leq c_{11} = 0 \\
			& {u_1}_1 - {u_1}_2 \leq c_{12} \\
			& {u_1}_2 - {u_1}_1 \leq c_{21} = c_{12} \\
			& {u_1}_2 - {u_1}_2 \leq c_{22} = 0
		\end{aligned}
	\end{equation*}
	The middle two constraints give
	$$|{u_1}_1 - {u_1}_2 | \leq c_{12}$$
	Recalling that ${u_1}_1 = 0$, we get that
	$$|{u_1}_2| \leq c_{12}$$
	If the number of support points was $N>2$, similar argument using constraints ${u_1}_1 - {u_1}_j \leq c_{1j}$ and ${u_1}_j - {u_1}_1 \leq c_{1j}$ would give
	$$|{u_1}_j| \leq c_{1j} \text{ for } j=1,\hdots,N$$
	Recall from Observation \ref{obs:cost_dist} that $c_{1j} = \frac{k-1}{k^2} \|x_1 - x_j \|^2$ which finishes the proof for $k=2$ measures.
	
	To see the result for $k>2$ measures, note that $A'u\leq c$ contains constraints 
	\begin{equation*}
		\begin{aligned}
			& \underbrace{{u_1}_1 + {u_2}_1 + \cdots + {u_{k-1}}_1}_{-{u_k}_1} + {u_k}_2 \leq c_{1\cdots 1 2} \\
			& \underbrace{{u_1}_2 + {u_2}_2 + \cdots + {u_{k-1}}_2}_{-{u_k}_2} + {u_k}_1 \leq c_{2\cdots 2 1} 
		\end{aligned}
	\end{equation*}
	and by Observation \ref{obs:cost_dist}, $c_{1\cdots 1 2} = c_{2\cdots 2 1}=\frac{k-1}{k^2}\|x_1 - x_2\|^2$
	giving 
	$$|{u_k}_2| \leq \frac{k-1}{k^2}\|x_1 - x_2\|^2$$
	and, by similar reasoning, 
	$$|{u_k}_j| \leq \frac{k-1}{k^2}\|x_1 - x_j\|^2  \text{ for }  j=1,\hdots,N$$
	Similarly, we conclude the same property for all dual variables $u_i$ indexed by $1,\hdots,k$, i.e. 
	$$|{u_i}_j| \leq \frac{k-1}{k^2}\|x_1 - x_j\|^2  \text{ for }  j=1,\hdots,N$$
	concluding the proof.
	
	\subsection{Proof of Lemma \ref{lem:clustered_alt}}\label{app:proof_clustered_alt}
	
	By the equivalence between $MOT$ and barycenter problems \eqref{eq:mot_general} and \eqref{eq:bary_general},
	$$MOT(\mu) = \inf_{\nu \in \mathcal{P}^2(\mathbb{R}^d)}\frac{1}{k}\sum_{i=1}^k W_2^2(\mu_i,\nu)$$
	$$= \inf_{\nu \in \mathcal{P}^2(\mathbb{R}^d)}\frac{1}{k} \left[ \frac{k}{2} W_2^2(\mu_1,\nu) + \frac{k}{2} W_2^2(\mu_k,\nu)\right] =  \inf_{\nu \in \mathcal{P}^2(\mathbb{R}^d)}\frac{1}{2}W_2^2(\mu_1,\nu) + \frac{1}{2}W_2^2(\mu_k,\nu)$$
	i.e. $\nu$ is a solution to the barycenter problem between $\mu_1$ and $\mu_k$ with optimal value $MOT(\mu_1,\mu_k) = \frac{1}{4}W_2^2(\mu_1,\mu_k)$. By similar reasoning, for $\mu \in \mathcal{F}^C_k$, the population value $MOT(\mu)$ is the same as the value of MOT computed with one measure from each cluster. 
	
	\subsection{Proof of Lemma \ref{lem:sparse_alt}}\label{app:proof_sparse_alt}
	Let $(u_1^*,u_k^*)$ be dual optimal solutions to the the problem $W_2^2(\mu_1,\mu_k)$. We will show that $(\frac{k-1}{k^2}u_1^*,0,\cdots,0, \frac{k-1}{k^2}u_k^*)$ are dual optimal for $MOT(\mu_1,\cdots,\mu_k)$, and hence $$MOT(\mu_1,\cdots,\mu_k)=\frac{k-1}{k^2}\langle u_1^*, \mu_1 \rangle + \frac{k-1}{k^2}\langle u_k^*, \mu_k \rangle = \frac{k-1}{k^2}W_2^2(\mu_1,\mu_k) $$
	By optimality of $(u_1^*,u_k^*)$ for $W_2^2(\mu_1,\mu_k)$, the dual constraints hold with equality
	$$u_1^{*i} + u_k^{*j} \leq \| x_i - x_j \|^2$$
	for all $i,j \in \{1,\hdots,N \}$, and hold with equality
	$$u_1^{*i} + u_k^{*j} = \| x_i - x_j \|^2$$
	on for some pairs $(i,j) \in I$ with the set $I$ indexing the pairs $(x_i,x_j)$ that support the optimal Wasserstein coupling $\pi^*$. Consider a multicoupling that agrees with $\pi^*$ on tuples $(x_i,\cdots,x_i,x_j)$, $(i,j) \in I$, and is zero otherwise (so the set of such tuples has a full mutlicoupling measure) - this is the candidate for the primal optimal solution to $MOT$. Further, the above equality implies, for $(i,j)\in I$,
	$$\frac{k-1}{k^2}u_1^{*i} + 0 + \cdots + 0 + \frac{k-1}{k^2} u_k^{*j} = \frac{k-1}{k^2}\| x_i - x_j \|^2 \underset{\text{Observation } \ref{obs:cost_dist}}{=} c_{i\cdots i j}$$
	Moreover, $c_{i\cdots i j}$ is no larger than the value of $c$ if some indices are not repeated, making the candidate $(\frac{k-1}{k^2}u_1^*,0,\cdots,0, \frac{k-1}{k^2}u_k^*)$ dual feasible for $MOT$. By complementary slackness (e. g., Lemma 1.1 of \cite{Gladkov} which specifically addresses the multimarginal problem), our dual candidate is optimal. 
	
	\subsection{Proof of Proposition \ref{prop:boot}}\label{app:proof_boot}
	\begin{itemize}
		\item[(a)] 
		By Theorem 3.1 of \cite{Hong}, the numerical directional derivative estimator $\frac{f(\hat{\mu}_n + \varepsilon_n \hat{G}^*) - f(\hat{\mu}_n)}{\varepsilon_n}$ is consistent for the directional derivative $f_\mu'(G)$ under mild measurability conditions on $\hat{G}^*$. The choice $\varepsilon_n = \frac{1}{r_m}$ with $m = \sqrt{n}$ (or, more generally, $m = n^p$, with $p \in (0,1)$) ensures that assumptions of the theorem are satisfied, i.e. that $\varepsilon_n \to 0$ and $\varepsilon_n r_n = \frac{r_n}{r_m} = \sqrt{\frac{n}{m}} \to \infty$ and allows to conclude consistency of this estimator for $f'_\mu(G)$. Note that consistency does not depend on the form of $f_\mu'(G)$, and hence holds under both $H_0$ and $H_a$.
		
		\item[(b)]

		We will check that the estimator $f'_n$ of the directional derivative map $f'_\mu$ given by $f'_n=f'_{\hat{\mu}_n}$ is uniformly consistent in the sense of Assumption 4 of \cite{Fang}. Note that under $H_0$, the estimator is given by
		\begin{equation*}
			\begin{aligned}
				f_{\hat{\mu}_n}':  h \to  & \max_{u} \sum_{i=1}^k \sqrt{a_i} \langle u_i, h_i \rangle \\
				& \sum_{i=1}^k u_i = 0, { } A'u \leq c\\
			\end{aligned}
		\end{equation*}
		The expression is independent of $\widehat{\mu}_n$, and hence the assumption is trivially satisfied. Thus, the proposed bootstrap is consistent by Theorem 3.2 of \cite{Fang}.  
	\end{itemize}
	\subsection{Details on the proof of Lemma \ref{lem:der_boot_compl}} \label{app:der_boot_comp} 
	Consider the linear program 
	$$\max_u \langle u, g\rangle \text{ s.t. } A'u \leq c \text{ and } Bu=0$$
	where $g \in \mathbb{R}^{kN}$ is a vector of objective coefficients, $\underset{N\times kN}{B}$ and $\underset{N^k\times kN}{A'}$ matrices with entries in $\{0,1 \}$, and $c$ a cost vector from primal MOT problem where the measures $\mu$ and support points in $\mathcal{X}$ are represented with $\log U$ bits of precision.
	
	Recall that a linear program over a polytope with exponentially many constraints can be solved in polynomial time by ellipsoid method if there exists a polynomial time separation oracle for the polytope (Theorem 8.5 of \cite{Bertsimas}). Here, we construct a separation oracle as follows. Given any $u \in \mathbb{R}^{kN}$, check if $N$ constraints $Bu=0$ are satisfied; if not, output a violated constraint (this is done with $\text{poly}(N)$ complexity). If $Bu=0$, check (and output a violated constraint if needed) in $A'u \leq c$ by employing a polynomial time oracle in Definition 10 of \cite{Altschuler}, which is done with $\text{poly}(N,k,\log U)$ complexity (Proposition 12 of \cite{Altschuler}).

	\section{Additional technical details and information on the data}\label{appB}
	\subsection{Additional technical details}\label{appB:Lip}
	
	\textbf{Locally Lipschitz property of $MOT$ functional.}
	Let $D=\bigotimes_{i=1}^k l^1(\mathcal{X})$ under $l^1$ distance $d$. 
	To establish the local Lipschitz property of the map $MOT: D \to \mathbb{R}$, we need to show that for every $\mu:=(\mu_1,\hdots,\mu_k) \in D$, there exists a constant $K > 0$ and a $\delta>0$ such for for all $\widetilde{\mu}:=(\widetilde{\mu}_1,\hdots,\widetilde{\mu}_k) \in D$ in a $\delta$-neighborhood of $\mu$, we have the Lipschitz behavior 
	\begin{equation*}
		\left|MOT(\mu) - MOT(\widetilde{\mu}) \right| \leq K \cdot d(\mu,\widetilde{\mu})
	\end{equation*}
	Note that the Lipschitz constant $K$ is allowed to depend on $\mu$.  
	
	Let $\nu$ and $\widetilde{\nu}$ denote the barycenters of collections $(\mu_1,\hdots, \mu_k)$ and $(\widetilde{\mu}_1,\hdots,\widetilde{\mu}_k)$, respectively. If $MOT(\mu) > MOT(\widetilde{\mu})$ (the other case will be treated similarly), we have that 
	\begin{eqnarray*}
		MOT(\mu) - MOT(\widetilde{\mu}) & = &  \frac{1}{k}  \sum_{i=1}^k W_2^2 (\mu_i,\nu) - 
		\frac{1}{k}\sum_{i=1}^k W_2^2(\widetilde{\mu_i},\widetilde{\nu})\\
		& \leq & \frac{1}{k}\sum_{i=1}^k    W_2^2(\mu_i,\widetilde{\nu}) -  \frac{1}{k}\sum_{i=1}^k W_2^2(\widetilde{\mu_i},\widetilde{\nu})   \\
	\end{eqnarray*}
	since $\widetilde{\nu}$ is suboptimal for the collection $(\mu_1,\hdots,\mu_k)$ and hence results in a larger objective value than an optimal barycenter $\nu$. Recall (see the proof of Theorem 4 in Supplementary Section C.1 of \cite{Sommerfeld}) that the map 
	\begin{equation*}
		\begin{aligned}
			& W_2^2(\cdot,\widetilde{\nu}): \xi \to W_2^2(\xi,\widetilde{\nu})
		\end{aligned}
	\end{equation*}
	is locally Lipschitz under the $l_1$ distance on its domain (which follows from bounding by total variation distance between the two discrete measures and observing its numerical equivalence with $l_1$ distance between the $N$-dimensional vectors representing these measures). The local Lipschitz constant obtained in \cite{Sommerfeld} depends on the argument $\xi$ via the diameter of the support of $\xi$. Using this result, we have that
	\begin{equation*}
		\begin{aligned}
			& MOT(\mu) - MOT(\widetilde{\mu})
			\leq \frac{1}{k}\sum_{i=1}^k    W_2^2(\mu_i,\widetilde{\nu}) - W_2^2(\widetilde{\mu_i},\widetilde{\nu})  \\ 
			& \leq \frac{1}{k} \sum_{i=1}^k K_i \cdot \|\mu_i - \widetilde{\mu}_i \|_{l_1} \leq \underbrace{\frac{1}{k} \max_i K_i}_{K} \cdot \underbrace{\sum_{i=1}^k \|\mu_i - \widetilde{\mu}_i \|_{l_1}}_{d(\mu,\widetilde{\mu})}
		\end{aligned}
	\end{equation*}
	The same bound holds for $MOT(\widetilde{\mu}) - MOT(\mu)$ 
	(proved the same way), and hence the map $MOT$ is indeed locally Lipschitz on $D$.

	\subsection{Information for \textit{SEER Tumor size} dataset}\label{appB:size}
	We consider the data from 2004 - 2015 on four cancer types: breast (female), prostate (male), lung (male), lung (female). For each cancer type, we wish to compare three measures $\mu_1,\mu_2,\mu_3$ corresponding to the tumor size distributions in following groups of patients: alive (or dead of causes other than the diagnosed cancer) at the end of the study with no distant metastasis at diagnosis, alive (or dead of causes other than the diagnosed cancer) at the end of the study with distant metastases at diagnosis, and dead at the end of the study (death is due to the diagnosed cancer). 
	
	Two classes of problems are considered: comparing distributions of tumor sizes $<1$ cm (measures supported on $\{1,2,\hdots,9 \}$ mm) and comparing distributions of the tumor sizes in the range $1 - 9 \text { cm}$ (measures supported on $\{1,2,\hdots,9 \}$ cm.) 
	
	The information from SEER used to download these data is given in Table 3. The sample sizes for each measure are given in Table 4. 
	
	\begin{table}[h!] \label{table:vars_sizes}
		\begin{center}
			\caption{Variables downloaded from SEER database to create \textbf{SEER Tumor size} dataset} 
			\resizebox{\textwidth}{!}{\begin{tabular}{||c |p{45mm} |p{45mm}||} 
					\hline
					\textbf{SEER*Stat variable name} & \textbf{variable indicates} & \textbf{codes values and their use}\\ [0.5ex] 
					\hline\hline
					SEER cause-specific death classification   & if alive (or dead from other cause)/dead (from this cancer) & \footnotesize{ - alive or dead of other cause} \newline \footnotesize{ - Dead (attributable to his cancer dx)}
					\\ 
					\hline
					CS mets at dx (2004 - 2015) & presence of distance mets &  - $00$: no distant mets \newline - (0,99): distant mets  \\
					\hline
					CS tumor size (2004 - 2015) & tumor size at diagnosis in mm\newline (the largest dimension) & use exact number of mm\newline (convert to cm and round to the nearest cm for $1 - 9$ cm tumors) \\[1ex]
					\hline
					
			\end{tabular}}
		\end{center}
	\end{table}
	
	\begin{table}[h!] \label{table:sampless_sizes}
		\begin{center}
			\caption{Sample sizes for \textbf{SEER Tumor size} dataset} 
			\begin{tabular}{||l ||c |c |c||} 
				\hline
				\textbf{Cancer type} & \textbf{alive, no mets} & \textbf{alive, mets} & \textbf{dead}\\ [0.5ex] 
				\hline\hline
				breast (female) $<1$ cm &  113,582  & 277 & 5,156 \\ 
				\hline
				prostate (male) $<1$ cm & 16,056 & group not included & 403 \\
				\hline
				lung (male) $<1$ cm &  1,513 & 187 &  1,905\\
				\hline
				lung (female) $<1$ cm & 2,735 & 231 &  2,178\\[1ex]
				\hline\hline
				breast (female) $1 - 9$ cm &  435,861 &  5,640 &  73,511\\ 
				\hline
				prostate (male) $1 - 9$ cm & 56,098 &  340 & 2,724 \\
				\hline
				lung (male) $1 - 9$ cm & 43,364 & 8,895 & 132,799 \\
				\hline
				lung (female) $1 - 9$ cm & 49,798 & 8,380 &  113,918\\[1ex]
				\hline
				
			\end{tabular}
		\end{center}
	\end{table}

	\subsection{Information for \textit{SEER Year of diagnosis} dataset}\label{appB:year}
	
	We consider the data on the lung and bronchus cancer in male and female in 2004 - 2020. For each sex, we wish to compare the bivariate distributions of patients' age and tumor sizes at diagnosis between four groups: patients diagnosed in $2004 - 2006$, $2009 - 2011$, $2014 - 2016$, and $2019 - 2020$. 
	\underline{Note:} Using the diagnosis years $2005$, $2010$, $2015$ and $2020$ produces similar results; additional years were included to increase sample sizes.
	
	The tumor size supports are restricted to $\{1, 3, 5 \}$ cm, which determine the boundaries between T1a/T1b, T1/T2, and T2/T3 in the TNM stage grouping of the non-small cell lung cancer\footnote{\url{https://www.cancer.org/cancer/types/lung-cancer/detection-diagnosis-staging/staging-nsclc.html}}. \normalsize The ages (given in 5 year increments by SEER) starting from 50 y.o. and ending at 80 y.o., which ensures enough data is available in each group. 
	
	The information from SEER used to download these data is given in Table 5. The sample sizes for each measure are given in Table 6.
	
	\begin{table}[h!] \label{table:vars_years}
		\begin{center}
			\caption{Variables downloaded from SEER database to create \textbf{SEER Tumor size} dataset} 
			\resizebox{\textwidth}{!}{\begin{tabular}{||l |p{45mm} |p{45mm}||} 
					\hline
					\textbf{SEER*Stat variable name} & \textbf{variable indicates} & \textbf{codes values and their use}\\ [0.5ex] 
					\hline\hline
					\hline
					sex & sex of a patient & select male or female from the drop-down menu \\
					\hline
					CS tumor size (2004 - 2015) & tumor size at diagnosis in mm\newline (the largest dimension)& use exact number of mm\newline (convert to cm and round to the nearest cm for $1 - 9$ cm tumors) \\
					\hline
					Age recode with $<1$ year olds & age at diagnosis in years& provided in 5-year intervals: \newline $<1, 1-4, \hdots, 85+ \text{years}$ \newline use lower bound\\[1ex]
					\hline
					
			\end{tabular}}
		\end{center}
	\end{table}
	
	\begin{table}[h!] \label{table:sampless_years}
		\begin{center}
			\caption{Sample sizes for \textbf{SEER Year of diagnosis} dataset} 
			\begin{tabular}{||l ||c |c |c |c ||} 
				\hline
				\textbf{Cancer type} & 2004 - 2006 & 2009 - 2011 & 2014 - 2016 & 2019 - 2020 \\ [0.5ex] 
				\hline\hline
				
				lung (male) & 7,971 & 8,668& 11,543& 10,439\\[1ex]
				\hline\hline
				
				lung (female) & 7,657& 9,106& 12,660 & 11,768\\[1ex]
				\hline
				
			\end{tabular}
		\end{center}
	\end{table}
\end{appendix}

\vspace{2 cm}
\begin{acks}[Acknowledgments]
	The author is grateful to the anonymous Referees, an Associate Editor and the Editor for their constructive comments that greatly improved the
	quality of this paper. The author is indebted to Ilmun Kim for sharing the codes for the tests used for comparisons in Figure \ref{fig:synth_3D}. The author sincerely thanks Adriana Dawes for the advice and support, Florian Gunsilius for the insightful comments, and Marc Hallin for suggesting relevant references. Helpful discussions with Helen Chamberlin, Jun Kitagawa, Facundo M{\'e}moli, Dustin Mixon, and Yulong Xing are gratefully acknowledged. The efforts of SEER Program in creation and maintenance of the SEER database are gratefully acknowledged. The author is solely responsible for all the mistakes. 
\end{acks}

\begin{funding}
	The author was supported by the National Institute of General Medical Science of the National Institutes of Health under award number R01GM132651 to Adriana Dawes.
\end{funding}

\bibliographystyle{imsart-number} 
\bibliography{reference.bib}       

%
%
%
%

\end{document}